\documentclass{article}

\usepackage{authblk}
\usepackage{a4wide}
\usepackage{amsthm}
\usepackage{graphicx}
\usepackage{amssymb}
\usepackage{amsmath}
\usepackage{ascmac}
\usepackage{setspace}
\usepackage{float}
\usepackage[dvips,usenames]{color}
\usepackage{colortbl}
\usepackage{algorithm}
\usepackage{algorithmic}
\usepackage{setspace}
\usepackage{tikz-cd}
\usepackage{comment}
\newtheorem{Theorem}[equation]{Theorem}

\newtheorem{Lemma}[equation]{Lemma}

\theoremstyle{definition}
\newtheorem{Definition}[equation]{Definition}

\theoremstyle{remark}

\numberwithin{equation}{section}

\DeclareMathOperator{\ev}{ev}

\DeclareMathOperator{\ad}{ad}

\newcommand{\ve}{\varepsilon}

\allowdisplaybreaks
\begin{document}
\title{The images of the higher generators via the evaluation map for the affine Yangian of type $A$}
\author{Mamoru Ueda\thanks{mueda@ms.u-tokyo.ac.jp}}
\date{}
\maketitle
\begin{abstract}
The affine Yangian associated with $\widehat{\mathfrak{sl}}(n)$ has several presentations: the current presentation, the minimalistic presentation and so on. The evaluation map for the affine Yangian was given by using the minimalistic presentation. One of the issues about the evaluation map is that the images of the evaluation maps are unkown except on finitely many generators.
In this article, we write down the images of the higher generators of the current presentation via the evaluation map for the affine Yangian of type $A$ explicitly.
\end{abstract}
\section{Introduction}
For a finite dimensional simple Lie algebra $\mathfrak{g}$, Drinfeld (\cite{D1},\cite{D2}) introduced the finite Yangian $Y_\hbar(\mathfrak{g})$ is a quantum group which is a deformation of the current algebra $\mathfrak{g}[u]$. The finite Yangian $Y_\hbar(\mathfrak{g})$ has a presentation whose generators $\{h_{i,r},x^\pm_{i,r}\}_{i\in I,r\geq0}$ correspond to $\{h_i\otimes u^r,x^\pm\otimes u^{r}\}_{i\in I,r\geq0}$, where $\{h_i,x^\pm_i\}_{i\in I}$ are Chevalley generators of $\mathfrak{g}$. This presentation is called the current presentation of the finite Yangian $Y_\hbar(\mathfrak{g})$. 
The Yangian associated with $\mathfrak{gl}(n)$ was also defined and we denote it by $Y_\hbar(\mathfrak{gl}(n))$. It is known that the finite Yangian $Y_\hbar(\mathfrak{gl}(n))$ is a deformation of the current algebra $\mathfrak{gl}[u]$. The Yangian $Y_\hbar(\mathfrak{gl}(n))$ has the RTT presentation whose generators are $\{T^{(r)}_{i,j}\}_{1\leq i,j\leq n,r\geq0}$, which corresponds to $\{\delta_{r,0}\delta_{i,j}+\delta(r\geq1)E_{i,j}u^{r-1}\}$ in $\mathfrak{gl}(n)[u]$. The embedding
\begin{equation*}
\widetilde{\iota}\colon Y_\hbar(\mathfrak{sl}(n))\to Y_\hbar(\mathfrak{gl}(n))
\end{equation*} 
was explicitly given by using the current presentation and the RTT presentation.

By using the current presentation, we can define the Yangian $Y_\hbar(\mathfrak{g})$ associated with a symmetrizable Kac-Moody Lie algebra $\mathfrak{g}$. As for the case that $\mathfrak{g}=\widehat{\mathfrak{sl}}(n)$, Guay (\cite{Gu2},\cite{Gu1}) defined the two parameter affine Yangian $Y_{\hbar,\ve}(\widehat{\mathfrak{sl}}(n))$, which is a deformation of the universal central extension of $\mathfrak{sl}(n)[u^{\pm1},v]$. Similarly to finite Yangians, the generators of the current presentation of $Y_{\hbar,\ve}(\widehat{\mathfrak{sl}}(n))$ are $\{h_{i,r},x^\pm_{i,r}\mid0\leq i\leq n-1,r\geq0\}$. In \cite{GNW}, Guay-Nakajima-Wendlandt gave a new presentation of $Y_{\hbar,\ve}(\widehat{\mathfrak{sl}}(n))$ whose generators are $\{h_{i,r},x^\pm_{i,r}\mid0\leq i\leq n-1,r=0,1\}$ and this presentation is called the minimalistic presentation. 

Recently, by using the minimalistic presentation, the relationships between affine Yangians and $W$-algebras have been studied. For instance, it was shown in \cite{U2} and \cite{KU} that there exists a surjective homomorphism 
\begin{equation*}
\Phi\colon Y_{\hbar,\ve}(\widehat{\mathfrak{sl}}(n))\to\mathcal{U}(\mathcal{W}^k(\mathfrak{gl}(ln),(l^n))),
\end{equation*}
where $\mathcal{U}(\mathcal{W}^k(\mathfrak{gl}(ln),(l^n)))$ is the universal enveloping algbera of a rectangular $W$-algebra associated with $\mathfrak{gl}(ln)$ and a nilpotent element of type $(l^n)$. It is expected that the universal enveloping algebra of a rectangular $W$-algebra of type $A$ can be written down as a quotient algebra of the affine Yangian $Y_{\hbar,\ve}(\widehat{\mathfrak{sl}}(n))$ via the homomorphism $\Phi$. One of the difficulties of this problem is that the images of $\Phi$ are unkown except for finite generators $\{h_{i,r},x^\pm_{i,r}\mid0\leq i\leq n-1,r=0,1\}$ since we construct $\Phi$ by using the minimalistic presentation of the affine Yangian.

In the case $l=1$, $\mathcal{U}(\mathcal{W}^k(\mathfrak{gl}(ln),(l^n)))$ coincides with the standard degreewise completion of the universal enveloping algebra of $\widehat{\mathfrak{gl}}(n)=\mathfrak{gl}(n)\otimes\mathbb{C}[u]\oplus\mathbb{C}c$. By using the minimalistic presentation, Kodera \cite{K1} gave a surjective homomorphism 
\begin{equation*}
\ev_{\hbar,\ve}\colon Y_{\hbar,\ve}(\widehat{\mathfrak{sl}}(n))\to\mathcal{U}(\widehat{\mathfrak{gl}}(n)),
\end{equation*}
where $\mathcal{U}(\widehat{\mathfrak{gl}}(n))$ is the standard degreewise completion of $U(\widehat{\mathfrak{sl}}(n))$. This homomrphism is called {\it the evaluation map}.  In this article, we give the images of $\{x^\pm_{i,r}\mid1\leq i\leq n-1,r\geq0\}$ via the evaluation map.

For this purpose, in Section~4, we introduce a new associative algebra $y_\hbar(\mathfrak{gl}(n))$ whose generators are $\{T^{(r)}_{i,j}\}_{1\leq i,j\leq n,r\geq0}$ with some of the defining relations of $Y_\hbar(\mathfrak{gl}(n))$. By the definition of $y_\hbar(\mathfrak{gl}(n))$, $Y_\hbar(\mathfrak{gl}(n))$ becomes a quotient algebra of $y_\hbar(\mathfrak{gl}(n))$. One of the features of $y_\hbar(\mathfrak{gl}(n))$ is the existence of an embedding
\begin{equation*}
\iota\colon Y_\hbar(\mathfrak{sl}(n))\to y_\hbar(\mathfrak{gl}(n)).
\end{equation*}
The images of $\{x^\pm_{i,r}\mid 1\leq i\leq n-2,r\geq0\}$ has the same formula as the homomorphism $\widetilde{\iota}$. Another feauture of $y_\hbar(\mathfrak{gl}(n))$ is the following theorem.
\begin{Theorem}
There exists a homomorphism
\begin{equation*}
\ev_\hbar\colon y_\hbar(\mathfrak{gl}(n))\to\mathcal{U}(\widehat{\mathfrak{gl}}(n))
\end{equation*}
determined by
\begin{align*}
\ev_\hbar(T^{(1)}_{i,j})&=E_{i,j},\\
\ev_\hbar(T^{(r)}_{i,j})&=\sum_{p=2}^m\limits\sum_{\substack{1\leq x_1,\cdots,x_{p}\leq n,\\z_1,\cdots,z_{p}\geq0}}\limits f^m_p((z_1+1)c,\cdots,(z_p+1)c)\\
&\qquad\qquad\qquad E_{i,x_1}t^{-z_1-1}E_{x_1,x_2}t^{z_1-z_2}\cdots E_{x_{p-1},x_p}t^{z_{p-1}-z_p}E_{x_p,j}t^{z_p+1},
\end{align*}
where we set a symmetric polynomial
\begin{align*}
f^m_p(z_1,\cdots,z_p)&=\prod_{1\leq i_1\leq\cdots\leq i_{m-p}\leq p}z_{i_1}\cdots z_{i_{m-p}}.
\end{align*}
\end{Theorem}
The subalgebra of the affine Yangian $Y_{\hbar,\ve}(\widehat{\mathfrak{sl}}(n))$ generated by $\{h_{i,r},x^\pm_{i,r}\mid 1\leq i\leq n-1,r\geq0\}$ is isomorphic to $Y_\hbar(\mathfrak{sl}(n))$.
We identify this subalgebra with $Y_\hbar(\mathfrak{sl}(n))$ and denote this subalgebra by $Y_\hbar(\mathfrak{sl}(n))$.
\begin{Theorem}\label{ref}
The following relation holds:
\begin{equation*}
\ev_\hbar\circ\iota=\ev_{\hbar,\ve}|_{Y_\hbar(\mathfrak{sl}(n))}.
\end{equation*}
In particular, we can write down $\{\ev_{\hbar,\ve}(x^\pm_{i,r})\mid 1\leq i\leq n-1,r\geq0\}$ explicitly.
\end{Theorem}
As for the finite Yangians, there exists a surjective homomorphism $\ev\colon Y_\hbar(\mathfrak{gl}(n))$ and the kernel of $\ev$ is generated by $\{T^{(r)}_{i,j}\mid r\geq2\}$.
We expect that the kernel of $\ev_{\hbar,\ve}$ and $\Phi$ can be obtained from this result.
\section*{Acknowledgement}
This research was supported by JSPS Research Fellowship for Young Scientists (PD), Grant Number JP25KJ0038.

\section{Affine Yangian of type $A$}
We recall the definition of the affine Yangian of type $A$. Let $(a_{i,j})_{0\leq i,j\leq n-1}$ be 
\begin{align*}
a_{i,j}&=\begin{cases}
2&\text{ if }i=j,\\
-1&\text{ if }i=j\pm1,\\
-1&\text{ if }(i,j)=(n-1,0),(0,n-1),\\
0&\text{ otherwise}.
\end{cases}
\end{align*}
The matrix $(a_{i,j})_{0\leq i,j\leq n-1}$ is the Cartan matrix of $\widehat{\mathfrak{sl}}(n)$.
\begin{Definition}[Definition~3.2 in \cite{Gu2} and Definition~2.3 in \cite{Gu1}]\label{Def32}
Suppose that $n\geq3$. For complex numbers $\hbar,\ve$, the affine Yangian $Y_{\hbar,\ve}(\widehat{\mathfrak{sl}}(n))$ is the associative algebra generated by 
\begin{equation*}
\{X^\pm_{i,r},H_{i,r}\mid 0\leq i\leq n-1,r\in\mathbb{Z}_{\geq0}\}
\end{equation*}
subject to the following defining relations:
\begin{gather}
[H_{i,r}, H_{j,s}] = 0,\label{Eq1.1}\\
[X_{i,r}^{+}, X_{j,s}^{-}] = \delta_{i,j} H_{i, r+s},\label{Eq1.2}\\
[H_{i,0}, X_{j,r}^{\pm}] = \pm a_{i,j} X_{j,r}^{\pm},\label{Eq1.4}\\
[H_{i, r+1}, X_{j, s}^{\pm}] - [H_{i, r}, X_{j, s+1}^{\pm}] = \pm a_{i,j}\dfrac{\hbar}{2} \{H_{i, r}, X_{j, s}^{\pm}\},\text{ if }(i,j)\neq(0,n-1),(n-1,0),\label{Eq1.5}\\
[H_{0, r+1}, X_{n-1, s}^{\pm}] - [H_{0, r}, X_{n-1, s+1}^{\pm}]= \mp\dfrac{\hbar}{2} \{H_{0, r}, X_{n-1, s}^{\pm}\}+(\ve+\dfrac{n}{2}\hbar) [H_{0, r}, X_{n-1, s}^{\pm}],\label{Eq1.6}\\
[H_{n-1, r+1}, X_{0, s}^{\pm}] - [H_{n-1, r}, X_{0, s+1}^{\pm}]= \mp\dfrac{\hbar}{2} \{H_{n-1, r}, X_{0, s}^{\pm}\}-(\ve+\dfrac{n}{2}\hbar) [H_{n-1, r}, X_{0, s}^{\pm}],\label{Eq1.7}\\
[X_{i, r+1}^{\pm}, X_{j, s}^{\pm}] - [X_{i, r}^{\pm}, X_{j, s+1}^{\pm}] = \pm a_{i,j}\dfrac{\hbar}{2} \{X_{i, r}^{\pm}, X_{j, s}^{\pm}\}\text{ if }(i,j)\neq(0,n-1),(n-1,0),\label{Eq1.8}\\
[X_{0, r+1}^{\pm}, X_{n-1, s}^{\pm}] - [X_{0, r}^{\pm}, X_{n-1, s+1}^{\pm}]= \mp\dfrac{\hbar}{2} \{X_{0, r}^{\pm}, X_{n-1, s}^{\pm}\} + (\ve+\dfrac{n}{2}\hbar) [X_{0, r}^{\pm}, X_{n-1, s}^{\pm}],\label{Eq1.9}\\
\sum_{\sigma\in S_{1-a_{i,j}}}\ad(X_{i,r_{\sigma(1)}}^{\pm})\cdots\ad(X_{i,r_{\sigma(1-a_{i,j})}}^{\pm})(X_{j,s}^{\pm})= 0 \ \text{ if }i \neq j, \label{Eq1.10}
\end{gather}
where $\{x,y\}=xy+yx$ and $S_{1-a_{i,j}}$ is a symmetric group of degree $1-a_{i,j}$.
\end{Definition}
One of the difficulties of Definition~\ref{Def32} is that the number of the generators is infinite. Guay-Nakajima-Wendlandt \cite{GNW} gave a presentation of the affine Yangian $Y_{\hbar,\ve}(\widehat{\mathfrak{sl}}(n))$ whose number of generators is finite.
\begin{Theorem}[Theorem 2.13 in \cite{GNW}]\label{Prop32}
Suppose that $n\geq3$. The affine Yangian $Y_{\hbar,\ve}(\widehat{\mathfrak{sl}}(n))$ is isomorphic to the associative algebra  generated by 
\begin{equation*}
\{X^\pm_{i,r},H_{i,r}\mid 0\leq i\leq n-1,r=0,1\}
\end{equation*} 
subject to the following defining relations:
\begin{gather}
[H_{i,r}, H_{j,s}] = 0,\label{Eq2.1}\\
[X_{i,0}^{+}, X_{j,0}^{-}] = \delta_{i,j} H_{i, 0},\label{Eq2.2}\\
[X_{i,1}^{+}, X_{j,0}^{-}] = \delta_{i,j} H_{i, 1} = [X_{i,0}^{+}, X_{j,1}^{-}],\label{Eq2.3}\\
[H_{i,0}, X_{j,r}^{\pm}] = \pm a_{i,j} X_{j,r}^{\pm},\label{Eq2.4}\\
[\tilde{H}_{i,1}, X_{j,0}^{\pm}] = \pm a_{i,j}\left(X_{j,1}^{\pm}\right),\text{ if }(i,j)\neq(0,n-1),(n-1,0),\label{Eq2.5}\\
[\tilde{H}_{0,1}, X_{n-1,0}^{\pm}] = \mp \left(X_{n-1,1}^{\pm}+(\ve+\dfrac{n}{2}\hbar) X_{n-1, 0}^{\pm}\right),\label{Eq2.6}\\
[\tilde{H}_{n-1,1}, X_{0,0}^{\pm}] = \mp \left(X_{0,1}^{\pm}-(\ve+\dfrac{n}{2}\hbar) X_{0, 0}^{\pm}\right),\label{Eq2.7}\\
[X_{i, 1}^{\pm}, X_{j, 0}^{\pm}] - [X_{i, 0}^{\pm}, X_{j, 1}^{\pm}] = \pm a_{ij}\dfrac{\hbar}{2} \{X_{i, 0}^{\pm}, X_{j, 0}^{\pm}\}\text{ if }(i,j)\neq(0,n-1),(n-1,0),\label{Eq2.8}\\
[X_{0, 1}^{\pm}, X_{n-1, 0}^{\pm}] - [X_{0, 0}^{\pm}, X_{n-1, 1}^{\pm}]= \mp\dfrac{\hbar}{2} \{X_{0, 0}^{\pm}, X_{n-1, 0}^{\pm}\} + (\ve+\dfrac{n}{2}\hbar) [X_{0, 0}^{\pm}, X_{n-1, 0}^{\pm}],\label{Eq2.9}\\
(\ad X_{i,0}^{\pm})^{1+|a_{i,j}|} (X_{j,0}^{\pm})= 0 \ \text{ if }i \neq j, \label{Eq2.10}
\end{gather}
where $\widetilde{H}_{i,1}=H_{i,1}-\dfrac{\hbar}{2}H_{i,0}^2$.
\end{Theorem}

\section{Evaluation map for the affine Yangian}
The evaluation map for the affine Yangian is a non-trivial homomorphism from the affine Yangian $Y_{\hbar,\ve}(\widehat{\mathfrak{sl}}(n))$ to the completion of the universal enveloping algebra of the affinization of $\mathfrak{gl}(n)$. Let us set a Lie algebra 
\begin{equation*}
\widehat{\mathfrak{gl}}(n)=\mathfrak{gl}(n)\otimes\mathbb{C}[z^{\pm1}]\oplus\mathbb{C}c\oplus\mathbb{C}z
\end{equation*}
whose commutator relations are given by
\begin{gather*}
[E_{i,j}t^s,E_{k,l}t^u]=\delta_{j,k}E_{i,l}t^{s+u}-\delta_{i,l}E_{k,j}t^{s+u}+s\delta_{s+u,0}\delta_{j,k}\delta_{i,l}c+s\delta_{s+u,0}\delta_{i,j}\delta_{k,l}z,\\
\text{$z$ and $c$ are central elements of }\widehat{\mathfrak{gl}}(n),
\end{gather*}
where tr is the trace of $\mathfrak{gl}(n)$, that is, $\text{tr}(E_{i,j}E_{k,l})=\delta_{i,l}\delta_{j,k}$. 
We consider a completion of $U(\widehat{\mathfrak{gl}}(n))/U(\widehat{\mathfrak{gl}}(n))(z-1)$ following \cite{MNT} and \cite{GNW}. 
We take the grading of $U(\widehat{\mathfrak{gl}}(n))/U(\widehat{\mathfrak{gl}}(n))(z-1)$ as $\text{deg}(Xt^s)=s$ and $\text{deg}(c)=0$. We denote the degreewise completion of $U(\widehat{\mathfrak{gl}}(n))/U(\widehat{\mathfrak{gl}}(n))(z-1)$ by $\mathcal{U}(\widehat{\mathfrak{gl}}(n))$.
\begin{Theorem}[Theorem 3.8 in \cite{K1}]\label{thm:main}
Suppose that $\hbar\neq0$ and $c =\dfrac{\ve}{\hbar}$.
Then, there exists an algebra homomorphism 
\begin{equation*}
\ev_{\hbar,\ve}\colon Y_{\hbar,\ve}(\widehat{\mathfrak{sl}}(n)) \to \mathcal{U}(\widehat{\mathfrak{gl}}(n))
\end{equation*}
uniquely determined by 
\begin{gather*}
\ev_{\hbar,\ve}(X_{i,0}^{+}) = \begin{cases}
E_{n,1}t&\text{ if }i=0,\\
E_{i,i+1}&\text{ if }1\leq i\leq n-1,
\end{cases} \ev_{\hbar,\ve}^n(X_{i,0}^{-}) = \begin{cases}
E_{1,n}t^{-1}&\text{ if }i=0,\\
E_{i+1,i}&\text{ if }1\leq i\leq n-1
\end{cases}
\end{gather*}
and
\begin{align*}
\ev_{\hbar,\ve}(H_{i,1}) &=-\dfrac{i-1}{2}\hbar \ev_{\hbar,\ve}(H_{i,0}) -\hbar E_{i,i}E_{i+1,i+1} \\
&\quad+ \hbar \displaystyle\sum_{s \geq 0}  \limits\displaystyle\sum_{k=1}^{i}\limits  E_{i,k}t^{-s}E_{k,i}t^s+\hbar \displaystyle\sum_{s \geq 0} \limits\displaystyle\sum_{k=i+1}^{n}\limits  E_{i,k}t^{-s-1}E_{k,i}t^{s+1}\\
&\quad-\hbar\displaystyle\sum_{s \geq 0}\limits\displaystyle\sum_{k=1}^{i}\limits E_{i+1,k}t^{-s} E_{k,i+1}t^{s}-\hbar\displaystyle\sum_{s \geq 0}\limits\displaystyle\sum_{k=i+1}^{n} \limits E_{i+1,k}t^{-s-1} E_{k,i+1}t^{s+1}\text{ for }i\neq0,\\
\ev_{\hbar,\ve}^n(X^+_{i,1}) &=-\dfrac{i-1}{2}\hbar \ev_{\hbar,\ve}(X^+_{i,0})+ \hbar \displaystyle\sum_{s \geq 0}  \limits\displaystyle\sum_{k=1}^{i}\limits  E_{i,k}t^{-s}E_{k,i+1}t^s\\
&\quad+\hbar \displaystyle\sum_{s \geq 0} \limits\displaystyle\sum_{k=i+1}^{n}\limits  E_{i,k}t^{-s-1}E_{k,i+1}t^{s+1}\text{ for }i\neq0,\\
\ev_{\hbar,\ve}^n(X^-_{i,1}) &=-\dfrac{i-1}{2}\hbar \ev_{\hbar,\ve}(X^-_{i,0})+ \hbar \displaystyle\sum_{s \geq 0}  \limits\displaystyle\sum_{k=1}^{i}\limits  E_{i+1,k}t^{-s}E_{k,i}t^s\\
&\quad+\hbar \displaystyle\sum_{s \geq 0} \limits\displaystyle\sum_{k=i+1}^{n}\limits  E_{i+1,k}t^{-s-1}E_{k,i}t^{s+1}\text{ for }i\neq0.
\end{align*}
\end{Theorem}
The proof was given by the presentation given in Theorem~\ref{Prop32} (see \cite{K1}).
\section{The finite Yangian of type $A$}
The one parameter Yangian associated with a symmetrizable Kac-Moody Lie algebra $\mathfrak{g}$ can be defined by the same way as Definition~\ref{Def32}. Here, we only recall the case that $\mathfrak{g}=\mathfrak{sl}(n)$.
\begin{Definition}\label{Def33}
Suppose that $n\geq3$. For a complex number $\hbar$, the finite Yangian $Y_{\hbar}(\mathfrak{sl}(n))$ is the associative algebra  generated by 
\begin{equation*}
\{X^\pm_{i,r},H_{i,r}\mid 1\leq i\leq n-1,r\in\mathbb{Z}_{\geq0}\}
\end{equation*}
subject to the defining relations \eqref{Eq1.1}-\eqref{Eq1.5}, \eqref{Eq1.8} and \eqref{Eq1.10}.
\end{Definition}
By Theorem 2.13 in \cite{GNW}, $Y_{\hbar}(\mathfrak{sl}(n))$ has a similar presentation to Theorem~\ref{Prop32}.
\begin{Theorem}[Theorem 2.13 in \cite{GNW}]\label{Prop33}
The finite Yangian $Y_{\hbar}(\widehat{\mathfrak{sl}}(n))$ is isomorphic to the associative algebra generated by $X_{i,r}^{+}, X_{i,r}^{-}, H_{i,r}$ $(i \in \{1,\cdots, n-1\}, r = 0,1)$ subject to the defining relations \eqref{Eq2.1}-\eqref{Eq2.5}, \eqref{Eq2.8} and \eqref{Eq2.10}.
\end{Theorem}
For the latter proof, we prepare one associative algebra by extending $Y_{\hbar}(\widehat{\mathfrak{sl}}(n))$.
\begin{Definition}[Theorem 2.13 in \cite{GNW}]
We define $y_{\hbar}(\widehat{\mathfrak{sl}}(n))$ as the associative algebra generated by $X_{i,r}^{+}, X_{i,r}^{-}, H_{i,r}$ $(i \in \{1,\cdots, n-1\}, r = 0,1)$ and $A_{n,1}$ subject to the defining relations \eqref{Eq2.1}-\eqref{Eq2.5}, \eqref{Eq2.8}, \eqref{Eq2.10} and
\begin{gather}
[\widetilde{H}_{n-1,1},A_{n,1}]=0,\label{Eq2.15}\\
[A_{n,1},X^\pm_{n,0}]=\mp X^\pm_{n,1}.\label{Eq2.16}
\end{gather}
\end{Definition}
By \eqref{Eq2.15} and \eqref{Eq2.16}, we have
\begin{align}
[A_{n,1},X^\pm_{n-1,r}]&=\dfrac{1}{2^r}[A_{n,1},\ad(\widetilde{H}_{n-1,1})^{r}X^\pm_{n-1,0}]\nonumber\\
&=\mp\dfrac{1}{2^r}\ad(\widetilde{H}_{n-1,1})^{r}X^\pm_{n-1,1}=\mp X^\pm_{n-1,r+1}.\label{Eq2.17}
\end{align}

Since $\mathfrak{gl}(n)$ is not a finite dimensional simple Lie algebra, we can not give the finite Yangian associated with $\mathfrak{gl}(n)$ by using Definition~\ref{Def33}.
The finite Yangian associated with $\mathfrak{gl}(n)$ is defined by using the RTT presentation.
\begin{Definition}
The finite Yangian associated with $\mathfrak{gl}(n)$ is an associative algebra generated by $\{T^{(r)}_{i,j}\mid r\geq0,1\leq i,j\leq n\}$ with the defining relations:
\begin{gather*}
T^{(0)}_{i,j}=\delta_{i,j},\\
[T^{(r+1)}_{i,j},T^{(s)}_{k,l}]-[T^{(r)}_{i,j},T^{(s+1)}_{k,l}]=-\hbar (T^{(r)}_{k,j}T^{(s)}_{i,l}-T^{(s)}_{k,j}T^{(r)}_{i,l}).
\end{gather*}
We denote the finite Yangian associated with $\mathfrak{gl}(n)$ by $Y_\hbar(\mathfrak{gl}(n))$.
\end{Definition}
Let us set the following element in $Y_\hbar(\mathfrak{gl}(n))[u^{-1}]$ for $1\leq a_1,\cdots,a_l, b_1,\cdots,b_l\leq n$:
\begin{align}
t^{a_1,\cdots,a_l}_{b_1,\cdots,b_l}(u)&=\sum_{\sigma,\tau\in S_l}\limits \text{sgn}(\sigma)\text{sgn}(\tau)t_{a_{\sigma(1)},b_{\tau(1)}}(u)t_{a_{\sigma(1)},b_{\tau(1)}}(u+\hbar)\cdots t_{a_{\sigma(l)},b_{\tau(l)}}(u+(l-1)\hbar),
\end{align}
where $t_{i,j}(u)=\sum_{r\geq0}\limits t^{(r)}_{i,j}u^{-r}$. 
The embedding
\begin{align*}
\iota\colon Y_{\hbar}(\mathfrak{sl}(n))\to y_\hbar(\mathfrak{gl}(n))
\end{align*}
was given by
\begin{align*}
\iota(X^+_{i,0})&=T^{(1)}_{i,i+1},\ \iota(X^-_{i,0})=T^{(1)}_{i+1,i},\\
\iota(\widetilde{H}_{i,1})&=T^{(2)}_{i,i}-T^{(2)}_{i+1,i+1}-\dfrac{i-1}{2}\hbar(T^{(1)}_{i,i}-T^{(1)}_{i+1,i+1})\\
&\quad+\hbar\sum_{u=1}^{i-1}T^{(1)}_{i,u}T^{(1)}_{u,i}-\hbar\sum_{u=1}^{i}T^{(1)}_{i+1,u}T^{(1)}_{u,i+1}+\dfrac{\hbar}{2}((T^{(1)}_{i,i})^2-(T^{(1)}_{i+1,i+1})^2).
\end{align*}

We introduce a new associative algebra of which $Y_\hbar(\mathfrak{gl}(n))$ becomes a quotient algebra.
\begin{Definition}
We define $y_\hbar(\mathfrak{gl}(n))$ as an associative algebra generated by $\{T^{(r)}_{i,j}\mid r\geq0,1\leq i,j\leq n\}$ with the defining relations:
\begin{gather}
T^{(0)}_{i,j}=\delta_{i,j},\ [T^{(1)}_{i,j},T^{(r)}_{k,l}]=\delta_{j,k}T^{(r)}_{i,l}-\delta_{i,l}T^{(r)}_{k,j},\label{ga1}\\
[T^{(2)}_{i,i},T^{(2)}_{j,j}]=-\hbar(T^{(1)}_{j,i}T^{(2)}_{i,j}-T^{(2)}_{j,i}T^{(1)}_{i,j}),\label{ga2}\\
[T^{(2)}_{i,i},T^{(r)}_{k,l}]=\delta_{i,k}T^{(r+1)}_{i,l}-\delta_{i,l}T^{(r+1)}_{k,i}-\hbar(T^{(1)}_{k,i}T^{(r)}_{i,l}-T^{(r)}_{k,i}T^{(1)}_{i,l}).\label{ga3}
\end{gather}
\end{Definition}
Similarly to $Y_\hbar(\mathfrak{gl}(n))$, we obtain the following theorem.
\begin{Theorem}
\textup{(1)}\ There exists an embedding
\begin{align*}
\iota\colon Y_{\hbar}(\mathfrak{sl}(n))\to y_\hbar(\mathfrak{gl}(n))
\end{align*}
given by
\begin{align*}
\iota(X^+_{i,0})&=T^{(1)}_{i,i+1},\ \iota(X^-_{i,0})=T^{(1)}_{i+1,i},\\
\iota(\widetilde{H}_{i,1})&=T^{(2)}_{i,i}-T^{(2)}_{i+1,i+1}-\dfrac{i-1}{2}\hbar(T^{(1)}_{i,i}-T^{(1)}_{i+1,i+1})\\
&\quad+\hbar\sum_{u=1}^{i-1}T^{(1)}_{i,u}T^{(1)}_{u,i}-\hbar\sum_{u=1}^{i}T^{(1)}_{i+1,u}T^{(1)}_{u,i+1}+\dfrac{\hbar}{2}((T^{(1)}_{i,i})^2-(T^{(1)}_{i+1,i+1})^2).
\end{align*}
\textup{(2)}\ There exists a homomorphism
\begin{align*}
\widehat{\iota}\colon y_{\hbar}(\mathfrak{sl}(n))\to y_\hbar(\mathfrak{gl}(n))
\end{align*}
given by
\begin{align*}
\widehat{\iota}(X^+_{i,0})&=T^{(1)}_{i,i+1},\ \widehat{\iota}(X^-_{i,0})=T^{(1)}_{i+1,i},\\
\widehat{\iota}(\widetilde{H}_{i,1})&=T^{(2)}_{i,i}-T^{(2)}_{i+1,i+1}-\dfrac{i-1}{2}\hbar(T^{(1)}_{i,i}-T^{(1)}_{i+1,i+1})\\
&\quad+\hbar\sum_{u=1}^{i-1}T^{(1)}_{i,u}T^{(1)}_{u,i}-\hbar\sum_{u=1}^{i}T^{(1)}_{i+1,u}T^{(1)}_{u,i+1}+\dfrac{\hbar}{2}((T^{(1)}_{i,i})^2-(T^{(1)}_{i+1,i+1})^2),\\
\widehat{\iota}(A_{n,1})&=T^{(2)}_{n,n}-\dfrac{n-2}{2}T^{(1)}_{n,n}+\hbar\sum_{u=1}^{n-1}T^{(1)}_{n,u}T^{(1)}_{u,n}+\dfrac{\hbar}{2}(T^{(1)}_{n,n})^2.
\end{align*}
\end{Theorem}
\begin{proof}
The proof is given by a direct computation. We only show the compatibility with \eqref{Eq2.15} and \eqref{Eq2.16}. The compatibilities of other relations can be proven by a similar way. By \eqref{ga1} and \eqref{ga2}, we have
\begin{align*}
&\quad[\widehat{\iota}(\widetilde{H}_{n-1,1}),\widehat{\iota}(A_{n,1})]\\
&=[T^{(2)}_{n-1,n-1},T^{(2)}_{n,n}]+[T^{(2)}_{n-1,n-1},\hbar\sum_{u=1}^{n-1}T^{(1)}_{n,u}T^{(1)}_{u,n}]\\
&\quad+[\hbar\sum_{u=1}^{n-2}T^{(1)}_{n-1,u}T^{(1)}_{u,n-1},T^{(2)}_{n,n}]+[\hbar\sum_{u=1}^{n-2}T^{(1)}_{n-1,u}T^{(1)}_{u,n-1},\hbar\sum_{u=1}^{n-1}T^{(1)}_{n,u}T^{(1)}_{u,n}]\\
&=-\hbar(T^{(1)}_{n,n-1}T^{(2)}_{n-1,n}-T^{(2)}_{n,n-1}T^{(1)}_{n-1,n})\\
&\quad-(\hbar T^{(2)}_{n,n-1}T^{(1)}_{n-1,n}-\hbar T^{(1)}_{n,n-1}T^{(2)}_{n-1,n})\\
&\quad+0-0\\
&=0.
\end{align*}
Thus, we have proved the compatibility with \eqref{Eq2.15}.

As for the compatibility with \eqref{Eq2.16}, we only show the plus case. By a direct computation, we obtain
\begin{align*}
&\quad[\widehat{\iota}(A_{n,1}),\widehat{\iota}(X^+_{n-1,0})]\\
&=[T^{(2)}_{n,n},T^{(1)}_{n-1,n}]-\dfrac{n-2}{2}[T^{(1)}_{n,n},T^{(1)}_{n-1,n}]+\hbar\sum_{u=1}^{n-1}[T^{(1)}_{n,u}T^{(1)}_{u,n},T^{(1)}_{n-1,n}]+\dfrac{\hbar}{2}[(T^{(1)}_{n,n})^2,T^{(1)}_{n-1,n}].\\
&=-T^{(2)}_{n-1,n}+\dfrac{n-1}{2}T^{(1)}_{n-1,n}-\hbar\sum_{u=1}^{n-1}T^{(1)}_{n-1,u}T^{(1)}_{u,n}+T^{(1)}_{n,n}T^{(1)}_{n-1,n}+\dfrac{\hbar}{2}\{T^{(1)}_{n,n},T^{(1)}_{n-1,n}\}\\
&=-T^{(2)}_{n-1,n}+\dfrac{n-1}{2}T^{(1)}_{n-1,n}-\hbar\sum_{u=1}^{n-1}T^{(1)}_{n-1,u}T^{(1)}_{u,n}.
\end{align*}
Thus, we have proved the compatibility with \eqref{Eq2.16} for the plus case.
\end{proof}

Let us set the following element in $y_\hbar(\mathfrak{gl}(n))[u^{-1}]$ for $1\leq a_1,\cdots,a_l,,b_1,\cdots,b_l\leq n$:
\begin{align}
t^{a_1,\cdots,a_l}_{b_1,\cdots,b_l}(u)&=\sum_{\sigma,\tau\in S_l}\limits \text{sgn}(\sigma)\text{sgn}(\tau)t_{a_{\sigma(1)},b_{\tau(1)}}(u)t_{a_{\sigma(1)},b_{\tau(1)}}(u+\hbar)\cdots t_{a_{\sigma(l)},b_{\tau(l)}}(u+(l-1)\hbar),\label{def31}
\end{align}
where $t_{i,j}(u)=\sum_{r\geq0}\limits t^{(r)}_{i,j}u^{-r}$. By the definition \eqref{def31}, we have
\begin{equation}
t^{a_1,\cdots,a_l}_{b_1,\cdots,b_l}(u)=0\label{al103-1}
\end{equation}
if $a_i=a_j$ or $b_i=b_j$ for some $i\neq j$.

Here after, we sometimes denote $\sum_{\text{condition}A}$ by $\sum$ in order to simplify the notation.
\begin{Lemma}
The following relations hold in the case that $a_1,\cdots,a_k$ and $b_1,\cdots,b_k$ are all different:
\begin{align}
&\quad[T^{(1)}_{i,j},t^{a_1,\cdots,a_k}_{b_1,\cdots,b_k}(v)]\nonumber\\
&=\sum_{u=1}^k\limits \delta_{a_u,j}t^{a_1,\cdots,i,\cdots,a_k}_{b_1,\cdots,b_k}(v)-\sum_{u=1}^k\limits t^{a_1,\cdots,a_k}_{b_1,\cdots,j,\cdots,b_k}(v)\delta_{i,b_u},\label{al100}\\
&\quad[T^{(2)}_{i,i},t^{a_1,\cdots,a_k}_{b_1,\cdots,b_k}(v)]-v[T^{(1)}_{i,i},t^{a_1,\cdots,a_k}_{b_1,\cdots,b_k}(v)]\nonumber\\
&=-\hbar\sum_{u=1}^k\limits T^{(1)}_{a_u,i}t^{a_1,\cdots,i,\cdots,a_k}_{b_1,\cdots,b_k}(v)+\hbar\sum_{u=1}^k\limits t^{a_1,\cdots,a_k}_{b_1,\cdots,i,\cdots,b_k}(v)T^{(1)}_{i,b_u}\nonumber\\
&\quad+\hbar t^{a_1,\cdots,a_k}_{b_1,\cdots,b_k}-\hbar\sum_{u,v=1}^k\limits \delta_{a_u,b_v}t^{a_1,\cdots,i,\cdots,a_k}_{b_1,\cdots,i,\cdots,b_k}\nonumber\\
\nonumber\\
&\quad-\hbar\sum_{\sigma,\tau\in S_l}\limits\sum_j (l+1-2j)\text{sgn}(\sigma)\text{sgn}(\tau)t_{a_{\sigma(1)},b_{\tau(1)}}(u)\nonumber\\
&\qquad\qquad\qquad\cdots (\delta_{a_{\sigma(j)},i}t_{i,b_{\tau(j)}}(u+(j-1)\hbar))\cdots t_{a_{\sigma(l)},b_{\tau(l)}}(u+(l-1)\hbar),\label{al101}\\
&\quad[T^{(2)}_{i,i},t^{a_1,\cdots,a_k}_{b_1,\cdots,b_k}(v)]-v[T^{(1)}_{i,i},t^{a_1,\cdots,a_k}_{b_1,\cdots,b_k}(v)]\nonumber\\
&=-\hbar\sum_{u=1}^k\limits t^{a_1,\cdots,i,\cdots,a_k}_{b_1,\cdots,b_k}(v)T^{(1)}_{a_u,i}+\hbar\sum_{u=1}^k\limits T^{(1)}_{i,b_u}t^{a_1,\cdots,a_k}_{b_1,\cdots,i,\cdots,b_k}(v)\nonumber\\
&\quad-\hbar\sum_{\sigma,\tau\in S_l}\limits\sum_j (l+1-2j)\text{sgn}(\sigma)\text{sgn}(\tau)t_{a_{\sigma(1)},b_{\tau(1)}}(u)\nonumber\\
&\qquad\qquad\qquad\cdots (\delta_{b_{\tau(j)},i}t_{a_{\sigma(j)},i}(u+(j-1)\hbar))\cdots t_{a_{\sigma(l)},b_{\tau(l)}}(u+(l-1)\hbar).\label{al103}
\end{align}
\end{Lemma}
\begin{proof}
The relation \eqref{al100} follows from the definition \eqref{def31}. We only show the relation \eqref{al101}. The relation \eqref{al103} can be proven by a direct computation. Here after, we denote the $i$-th term of the equation $(\text{equation number})$ by $(\text{equation number})_i$. By the definition, we have
\begin{align}
&\quad[T^{(2)}_{i,i},t^{a_1,\cdots,a_l}_{b_1,\cdots,b_l}(v)]-v[T^{(1)}_{i,i},t^{a_1,\cdots,a_l}_{b_1,\cdots,b_l}(v)]\nonumber\\
&=-\hbar\sum_{\sigma,\tau\in S_l}\limits \text{sgn}(\sigma)\text{sgn}(\tau)t_{a_{\sigma(1)},b_{\tau(1)}}(u)\nonumber\\
&\qquad\qquad\cdots (T^{(1)}_{a_{\sigma(j)},i}t_{i,b_{\tau(j)}}(u+(j-1)\hbar)-t_{a_{\sigma(j)},i}(u+(j-1)\hbar)T^{(1)}_{i,b_{\tau(j)}})\cdots t_{a_{\sigma(l)},b_{\tau(l)}}(u+(l-1)\hbar)\nonumber\\
&\quad+(j-1)\hbar\sum_{\sigma,\tau\in S_l}\limits \text{sgn}(\sigma)\text{sgn}(\tau)t_{a_{\sigma(1)},b_{\tau(1)}}(u)\nonumber\\
&\qquad\qquad\cdots (\delta_{a_{\sigma(j)},i}t_{i,b_{\tau(j)}}(u+(j-1)\hbar)-\delta_{b_{\tau(j)},i}t_{a_{\sigma(j)},i}(u+(j-1)\hbar))\cdots t_{a_{\sigma(l)},b_{\tau(l)}}(u+(l-1)\hbar)\nonumber\\
&=-\hbar\sum_{\sigma,\tau\in S_l}\limits \sum_j\text{sgn}(\sigma)\text{sgn}(\tau)t_{a_{\sigma(1)},b_{\tau(1)}}(u)\nonumber\\
&\qquad\qquad\qquad\cdots (T^{(1)}_{a_{\sigma(j)},i}t_{i,b_{\tau(j)}}(u+(j-1)\hbar))\cdots t_{a_{\sigma(l)},b_{\tau(l)}}(u+(l-1)\hbar)\nonumber\\
&\quad+\hbar\sum_{\sigma,\tau\in S_l}\limits\sum_j \text{sgn}(\sigma)\text{sgn}(\tau)t_{a_{\sigma(1)},b_{\tau(1)}}(u)\nonumber\\
&\qquad\qquad\qquad\cdots (t_{a_{\sigma(j)},i}(u+(j-1)\hbar)T^{(1)}_{i,b_{\tau(j)}}\cdots t_{a_{\sigma(l)},b_{\tau(l)}}(u+(l-1)\hbar)\nonumber\\
&\quad+\hbar\sum_{\sigma,\tau\in S_l}\limits \text{sgn}(\sigma)\text{sgn}(\tau)(j-1)t_{a_{\sigma(1)},b_{\tau(1)}}(u)\nonumber\\
&\qquad\qquad\qquad\cdots (\delta_{a_{\sigma(j)},i}t_{i,b_{\tau(j)}}(u+(j-1)\hbar))\cdots t_{a_{\sigma(l)},b_{\tau(l)}}(u+(l-1)\hbar)\nonumber\\
&\quad-\hbar\sum_{\sigma,\tau\in S_l}\limits \text{sgn}(\sigma)\text{sgn}(\tau)(j-1)t_{a_{\sigma(1)},b_{\tau(1)}}(u)\nonumber\\
&\qquad\qquad\qquad\cdots (\delta_{b_{\tau(j)},i}t_{a_{\sigma(j)},i}(u+(j-1)\hbar))\cdots t_{a_{\sigma(l)},b_{\tau(l)}}(u+(l-1)\hbar).\label{al200}
\end{align}
By a direct computation, we obtain
\begin{align}
\eqref{al200}_1&=-\hbar\sum_{\sigma,\tau\in S_l}\limits \sum_j\text{sgn}(\sigma)T^{(1)}_{a_{\sigma(j)},i}\text{sgn}(\tau)t_{a_{\sigma(1)},b_{\tau(1)}}(u)\nonumber\\
&\qquad\qquad\qquad\cdots \cdots (t_{i,b_{\tau(j)}}(u+(j-1)\hbar))\cdots t_{a_{\sigma(l)},b_{\tau(l)}}(u+(l-1)\hbar)\nonumber\\
&\quad-\hbar\sum_{\sigma,\tau\in S_l}\limits \sum_{x<j}\text{sgn}(\sigma)\text{sgn}(\tau)t_{a_{\sigma(1)},b_{\tau(1)}}(u)\cdots [t_{a_{\sigma(x)},b_{\tau(x)}}(u+(j-1)\hbar),T^{(1)}_{a_{\sigma(j)},i}]\nonumber\\
&\qquad\qquad\qquad\cdots t_{i,b_{\tau(j)}}(u+(j-1)\hbar))\cdots t_{a_{\sigma(l)},b_{\tau(l)}}(u+(l-1)\hbar)\nonumber\\
&=-\hbar\sum_{\sigma,\tau\in S_l}\limits \sum_j\text{sgn}(\sigma)T^{(1)}_{a_{\sigma(j)},i}\text{sgn}(\tau)t_{a_{\sigma(1)},b_{\tau(1)}}(u)\nonumber\\
&\qquad\qquad\qquad\cdots (t_{i,b_{\tau(j)}}(u+(j-1)\hbar)\cdots t_{a_{\sigma(l)},b_{\tau(l)}}(u+(l-1)\hbar)\nonumber\\
&\quad-\hbar\sum_{\sigma,\tau\in S_l}\limits \sum_{x<j}\text{sgn}(\sigma)\text{sgn}(\tau)t_{a_{\sigma(1)},b_{\tau(1)}}(u)\cdots (\delta_{b_{\tau(x)},a_{\sigma(j)}}t_{a_{\sigma(x)},i}(u+(x-1)\hbar))\nonumber\\
&\qquad\qquad\qquad\cdots t_{i,b_{\tau(j)}}(u+(j-1)\hbar)\cdots t_{a_{\sigma(l)},b_{\tau(l)}}(u+(l-1)\hbar)\nonumber\\
&\quad+\hbar\sum_{\sigma,\tau\in S_l}\limits \sum_{x<j}\text{sgn}(\sigma)\text{sgn}(\tau)t_{a_{\sigma(1)},b_{\tau(1)}}(u)\cdots(\delta_{a_{\sigma(x)},i}t_{a_{\sigma(j)},b_{\tau(x)}}(u+(x-1)\hbar))\nonumber\\
&\qquad\qquad\qquad\cdots t_{i,b_{\tau(j)}}(u+(j-1)\hbar)\cdots t_{a_{\sigma(l)},b_{\tau(l)}}(u+(l-1)\hbar)\nonumber\\
&=-\hbar\sum_{u=1}^k\limits T^{(1)}_{a_u,i}t^{a_1,\cdots,i,\cdots,a_k}_{b_1,\cdots,b_k}(v)\nonumber\\
&\quad-\hbar\sum_{\sigma,\tau\in S_l}\limits \sum_{x<j}\text{sgn}(\sigma)\text{sgn}(\tau)t_{a_{\sigma(1)},b_{\tau(1)}}(u)\cdots(\delta_{b_{\tau(x)},a_{\sigma(j)}}t_{a_{\sigma(x)},i}(u+(x-1)\hbar))\nonumber\\
&\qquad\qquad\qquad\cdots t_{i,b_{\tau(j)}}(u+(j-1)\hbar)\cdots t_{a_{\sigma(l)},b_{\tau(l)}}(u+(l-1)\hbar)\nonumber\\
&\quad-\hbar\sum_{\sigma,\tau\in S_l}\limits\sum_j (j-1)\text{sgn}(\sigma)\text{sgn}(\tau)t_{a_{\sigma(1)},b_{\tau(1)}}(u)\nonumber\\
&\qquad\qquad\qquad\cdots (\delta_{b_{\tau(j)},i}t_{a_{\sigma(j)},i}(u+(j-1)\hbar))\cdots t_{a_{\sigma(l)},b_{\tau(l)}}(u+(l-1)\hbar).\label{al400}\\
\eqref{al200}_2&=\hbar\sum_{\sigma,\tau\in S_l}\limits\sum_j \text{sgn}(\sigma)\text{sgn}(\tau)t_{a_{\sigma(1)},b_{\tau(1)}}(u)\nonumber\\
&\qquad\qquad\qquad\cdots t_{a_{\sigma(j)},i}(u+(j-1)\hbar)\cdots t_{a_{\sigma(l)},b_{\tau(l)}}(u+(l-1)\hbar)T^{(1)}_{i,b_{\tau(j)}}\nonumber\\
&\quad+\hbar\sum_{\sigma,\tau\in S_l}\limits\sum_{j<x} \text{sgn}(\sigma)\text{sgn}(\tau)t_{a_{\sigma(1)},b_{\tau(1)}}(u)\cdots (t_{a_{\sigma(j)},i}(u+(j-1)\hbar))\nonumber\\
&\qquad\qquad\qquad\cdots [T^{(1)}_{i,b_{\tau(j)}},t_{a_{\sigma(x)},b_{\sigma(x)}}(u+(x-1)\hbar)]\cdots t_{a_{\sigma(l)},b_{\tau(l)}}(u+(l-1)\hbar)\nonumber\\
&=\hbar\sum_{\sigma,\tau\in S_l}\limits\sum_j \text{sgn}(\sigma)\text{sgn}(\tau)t_{a_{\sigma(1)},b_{\tau(1)}}(u)\nonumber\\
&\qquad\qquad\qquad\cdots (t_{a_{\sigma(j)},i}(u+(j-1)\hbar))\cdots t_{a_{\sigma(l)},b_{\tau(l)}}(u+(l-1)\hbar)T^{(1)}_{i,b_{\tau(j)}}\nonumber\\
&\quad+\hbar\sum_{\sigma,\tau\in S_l}\limits\sum_{j<x} \text{sgn}(\sigma)\text{sgn}(\tau)t_{a_{\sigma(1)},b_{\tau(1)}}(u)\cdots (t_{a_{\sigma(j)},i}(u+(j-1)\hbar)\nonumber\\
&\qquad\qquad\qquad\cdots (\delta_{a_{\sigma(x)},b_{\tau}(j)}t_{i,\sigma(x)}(u+(x-1)\hbar))\cdots t_{a_{\sigma(l)},b_{\tau(l)}}(u+(l-1)\hbar)\nonumber\\
&\quad-\hbar\sum_{\sigma,\tau\in S_l}\limits\sum_{j<x} \text{sgn}(\sigma)\text{sgn}(\tau)t_{a_{\sigma(1)},b_{\tau(1)}}(u)\cdots (\delta_{b_{\sigma(x)},i}t_{a_{\sigma(x)},b_{\tau(j)}}(u+(j-1)\hbar))\nonumber\\
&\qquad\qquad\qquad\cdots (t_{a_{\sigma(j)},i}(u+(j-1)\hbar))\cdots t_{a_{\sigma(l)},b_{\tau(l)}}(u+(l-1)\hbar)\nonumber\\
&=\hbar\sum_{u=1}^k\limits t^{a_1,\cdots,a_k}_{b_1,\cdots,j,\cdots,b_k}(v)T^{(1)}_{i,b_u}\nonumber\\
&\quad+\hbar\sum_{\sigma,\tau\in S_l}\limits\sum_{j<x} \text{sgn}(\sigma)\text{sgn}(\tau)t_{a_{\sigma(1)},b_{\tau(1)}}(u)\cdots (t_{a_{\sigma(j)},i}(u+(j-1)\hbar)\nonumber\\
&\qquad\qquad\qquad\cdots (\delta_{a_{\sigma(x)},b_{\tau}(j)}t_{i,\sigma(x)}(u+(x-1)\hbar))\cdots t_{a_{\sigma(l)},b_{\tau(l)}}(u+(l-1)\hbar)\nonumber\\
&\quad+\hbar\sum_{\sigma,\tau\in S_l}\limits\sum_j (l-j)\text{sgn}(\sigma)\text{sgn}(\tau)t_{a_{\sigma(1)},b_{\tau(1)}}(u)\nonumber\\
&\qquad\qquad\qquad\cdots (\delta_{b_{\tau(j)},i}t_{a_{\sigma(j)},i}(u+(j-1)\hbar))\cdots t_{a_{\sigma(l)},b_{\tau(l)}}(u+(l-1)\hbar).\label{al401}
\end{align}
Since $\eqref{al200}_3+\eqref{al401}_3=0$ and $\eqref{al400}_2+\eqref{al401}_3=0$ hold, we obtain the relation \eqref{al101}.
\end{proof}
\begin{Theorem}
For $1\leq i\leq n-1$, we have
\begin{gather}
\iota(x^+_i(u))=t^{1,\cdots,i}_{1,\cdots,i}(u+(i-1)\dfrac{\hbar}{2})^{-1}t_{1,\cdots,i-1,i+1}^{1,\cdots,i}(u+\dfrac{\hbar}{2}(i-1)),\label{ee1}\\
\iota(x^-_i(u))=t^{1,\cdots,i-1,i+1}_{1,\cdots,i}(u+(i-1)\dfrac{\hbar}{2})t^{1,\cdots,i}_{1,\cdots,i}(u+(i-1)\dfrac{\hbar}{2})^{-1},\label{ee2}
\end{gather}
where $x^+_1(u)=\sum_{r\geq0}\limits x^+_{1,r}u^{-r-1}$ and $t_{i,j}(u)=\delta_{i,j}+\sum_{r\geq0}\limits T^{(r)}_{i,j}u^{-r-1}$.
\end{Theorem}
\begin{proof}
The well-definedness of $\iota$ follows from Theorem~\ref{Prop33} and a direct computation. We only show the relation \eqref{ee1}. As for the case $1\leq i\leq n-2$, it is enough to show the following relations:
\begin{gather}
[\iota(\widetilde{H}_{i+1,1}),\iota(x^+_i(u))]=-u\iota(x^+_i(u))+\hbar t^{1,\cdots,i}_{1,\cdots,i}(u+(i-1)\dfrac{\hbar}{2})T^{(1)}_{i,i+1},\label{ga112}\\
[\widehat{\iota}(A_{n,1}),\widehat{\iota}(x^+_{n-1}(u))]=u\widehat{\iota}(x^+_{n-1}(u))-\hbar t^{1,\cdots,n-1}_{1,\cdots,n-1}(u+(i-1)\dfrac{\hbar}{2})T^{(1)}_{n-1,n}.\label{ga111}
\end{gather}
The both relations can be derived by a direct computation. We only show \eqref{ga112}. The relation \eqref{ga111} can be proven by a similar way. Since
\begin{align*}
\iota(\widetilde{H}_{i+1,1})&=T^{(2)}_{i+1,i+1}-T^{(2)}_{i+2,i+2}+\dfrac{i+2}{2}\hbar (T^{(1)}_{i+1,i+1}-T^{(1)}_{i+2,i+2})\\
&\quad+\hbar\sum_{v=1}^{i}T^{(1)}_{v,i+1}T^{(1)}_{i+1,v}-\hbar\sum_{v=1}^{i+1}T^{(1)}_{v,i+2}T^{(1)}_{i+2,v}+\dfrac{\hbar}{2}((T^{(1)}_{i+1,i+1})^2-(T^{(1)}_{i+2,i+2})^2).
\end{align*}
holds, we have
\begin{align*}
&\quad[\iota(\widetilde{H}_{i+1,1}),t^{1,\cdots,i}_{1,\cdots,i}(u+(i-1)\dfrac{\hbar}{2})]\\
&=[T^{(2)}_{i+1,i+1}+\hbar\sum_{v\leq i} T^{(1)}_{v,i+1}T^{(1)}_{i+1,v},t^{1,\cdots,i}_{1,\cdots,i}(u+(i-1)\dfrac{\hbar}{2})]\\
&\quad-[T^{(2)}_{i+2,i+2}+\hbar\sum_{v\leq i+1} T^{(1)}_{v,i+2}T^{(1)}_{i+2,v},t^{1,\cdots,i}_{1,\cdots,i}(-u+(i-1)\dfrac{\hbar}{2}),t^{1,\cdots,i}_{1,\cdots,i}(u+(i-1)\dfrac{\hbar}{2})]\\
&\quad+\dfrac{i+2}{2}\hbar[(T^{(1)}_{i+1,i+1}-T^{(1)}_{i+2,i+2}),t^{1,\cdots,i}_{1,\cdots,i}(u+(i-1)\dfrac{\hbar}{2})]\\
&\quad+[\dfrac{\hbar}{2}((T^{(1)}_{i+1,i+1})^2-(T^{(1)}_{i+2,i+2})^2),t^{1,\cdots,i}_{1,\cdots,i}(u+(i-1)\dfrac{\hbar}{2})].
\end{align*}
Since all of the terms are equal to zero by \eqref{al100} and \eqref{al101}, we have
\begin{equation*}
[\iota(\widetilde{H}_{i-1,1}),t^{1,\cdots,i}_{1,\cdots,i}(u+(i-1)\dfrac{\hbar}{2})]=0.
\end{equation*}
Similarly, we obtain
\begin{align*}
&\quad[\iota(\widetilde{H}_{i-1,1}),t^{1,\cdots,i}_{1,\cdots,i-1,i+1}(u+(i-1)\dfrac{\hbar}{2})]\\
&=[T^{(2)}_{i+1,i+1}+\hbar\sum_{v\leq i} T^{(1)}_{i+1,v}T^{(1)}_{v,i+1},t^{1,\cdots,i}_{1,\cdots,i-1,i+1}(u+(i-1)\dfrac{\hbar}{2})]\\
&\quad-[T^{(2)}_{i+2,i+2}+\hbar\sum_{v\leq i+1} T^{(1)}_{i+2,v}T^{(1)}_{v,i+2},t^{1,\cdots,i}_{1,\cdots,i}(-u+(i-1)\dfrac{\hbar}{2}),t^{1,\cdots,i}_{1,\cdots,i-1,i+1}(u+(i-1)\dfrac{\hbar}{2})]\\
&\quad-\dfrac{i}{2}\hbar[(T^{(1)}_{i+1,i+1}-T^{(1)}_{i+2,i+2}),t^{1,\cdots,i}_{1,\cdots,i-1,i+1}(u+(i-1)\dfrac{\hbar}{2})]\\
&\quad+[\dfrac{\hbar}{2}((T^{(1)}_{i+1,i+1})^2-(T^{(1)}_{i+2,i+2})^2),t^{1,\cdots,i}_{1,\cdots,i-1,i+1}(u+(i-1)\dfrac{\hbar}{2})].
\end{align*}
By \eqref{al100} and \eqref{al103-1}, we obtain
\begin{align*}
&\quad-[T^{(2)}_{i+2,i+2}+\hbar\sum_{v\leq i+1} T^{(1)}_{i+2,v}T^{(1)}_{v,i+2}t^{1,\cdots,i}_{1,\cdots,i}(u+(i-1)\dfrac{\hbar}{2}),t^{1,\cdots,i}_{1,\cdots,i-1,i+1}(u+(i-1)\dfrac{\hbar}{2})]\\
&=0,\\
&\quad[T^{(2)}_{i+1,i+1}+\hbar\sum_{v\leq i} T^{(1)}_{i+1,v}T^{(1)}_{v,i+1},t^{1,\cdots,i}_{1,\cdots,i-1,i+1}(u+(i-1)\dfrac{\hbar}{2})]\\
&=-(u+\dfrac{\hbar}{2}(i-1))t^{1,\cdots,i}_{1,\cdots,i-1,i+1}(u+(i-1)\dfrac{\hbar}{2})+\hbar T^{(1)}_{i+1,i+1}t^{1,\cdots,i}_{1,\cdots,i-1,i+1}(u+(i-1)\dfrac{\hbar}{2}),\\
&\quad-\dfrac{i}{2}\hbar[(T^{(1)}_{i+1,i+1}-T^{(1)}_{i+2,i+2}),t^{1,\cdots,i}_{1,\cdots,i-1,i+1}(u+(i-1)\dfrac{\hbar}{2})]\\
&=\dfrac{\hbar}{2}it^{1,\cdots,i}_{1,\cdots,i-1,i+1}(-u+(i-1)\dfrac{\hbar}{2}),\\
&\quad[\hbar\sum_{v\leq i-2} T^{(1)}_{v,i-1}T^{(1)}_{i-1,v},t^{1,\cdots,i}_{1,\cdots,i-1,i+1}(u+(i-1)\dfrac{\hbar}{2})]\\
&=0,\\
&\quad+[\dfrac{\hbar}{2}((T^{(1)}_{i+1,i+1})^2-(T^{(1)}_{i+2,i+2})^2),t^{1,\cdots,i}_{1,\cdots,i-1,i+1}(u+(i-1)\dfrac{\hbar}{2})]\\
&=-\dfrac{\hbar}{2}\{T^{(1)}_{i+1,i+1},t^{1,\cdots,i}_{1,\cdots,i-1,i+1}(u+(i-1)\dfrac{\hbar}{2})\}.
\end{align*}
Thus, we find the relation \eqref{ga112} for $1\leq i\leq n-2$.

\end{proof}

\section{The images of the higher terms via the evaluation map for the affine Yangian of type $A$}
Let us take a comletely symmetric polynomial:
\begin{equation*}
h_{m}(z_1,\cdots,z_n)=\prod_{1\leq i_1\leq\cdots\leq i_{m}\leq n}z_{i_1}\cdots z_{i_m}.
\end{equation*}
By the definition, the following relation holds:
\begin{equation}
h_{m}(z_1,\cdots,z_{n+1})=h_m(z_1,\cdots,z_n)+z_{n+1}h_{m-1}(z_1,\cdots,z_{n+1}).\label{rel0}
\end{equation}
For a complex number $c$ and positive integers $m,n$, we also set a polynomial
\begin{equation*}
f^m_n(z_1,\cdots,z_n)=\begin{cases}
h_{m-n}((z_1+1)c,\cdots,(z_n+1)c)&\text{ if }2\leq n\leq m,\\
0&\text{ otherwise}.
\end{cases}
\end{equation*}
\begin{Lemma}\label{Lem126}
The following relations hold:
\begin{gather}
f^m_{n+1}(z_1,\cdots,z_{n+1})=f^m_n(z_1,\cdots,z_n)+(z_{n+1}+1)cf^{m-1}_{n+1}(z_1,\cdots,z_{n+1}),\label{rel1}\\
f^m_n(z_1,\cdots,z_{n-1},z_n+a)-f^m_n(z_1,\cdots,z_{n-1},z_n)=acf^m_{n+1}(z_1,\cdots,z_n,z_n+a)\label{rel2}
\end{gather}
for a complex number $a$.
\end{Lemma}
\begin{proof}
The relation \eqref{rel1} follows from \eqref{rel0}. We show the relation \eqref{rel2} by the induction hypothesis on $m$. The case $m=1$ follows from the definition of $f^m_n$. Suppose that \eqref{rel2} follows in the case $m\leq k-1$.
By \eqref{rel0} and \eqref{rel1}, we have
\begin{align*}
&\quad f^k_n(z_1,\cdots,z_{n-1},z_n=z_n+a)-f^k_n(z_1,\cdots,z_{n-1},z_n)\\
&=f^{k}_{n-1}(z_1,\cdots,z_{n-1})+(z_n+a+1)cf^{k-1}_n(z_1,\cdots,z_{n-1},z_n+a)\\
&\quad-f^{k}_{n-1}(z_1,\cdots,z_{n-1})-(z_n+1)cf^{k-1}_n(z_1,\cdots,z_{n-1},z_n)\\
&=acf^{k-1}_n(z_1,\cdots,z_{n-1},z_n+a)+(z_n+1)c(f^{k-1}_n(z_1,\cdots,z_{n-1},z_n+a)-f^{k-1}_n(z_1,\cdots,z_{n-1},z_n))\\
&=acf^{k-1}_n(z_1,\cdots,z_{n-1},z_n+a)+(z_n+1)acf^{k-1}_{n+1}(z_1,\cdots,z_{n-1},z_n+a)\\
&=acf^{k}_{n+1}(z_{n+1}=z_n+a),
\end{align*}
where the third equality is derived from the induction hypothesis.
\end{proof}
\begin{Theorem}\label{Main}
There exists a homomorphism
\begin{gather*}
\ev_\hbar\colon y_\hbar(\mathfrak{gl}(p))\to \mathcal{U}(\widehat{\mathfrak{gl}}(p))
\end{gather*}
given by
\begin{align*}
\ev_\hbar(T^{(1)}_{i,j})&=E_{i,j},\ \ev_\hbar(T^{(0)}_{i,j})=\delta_{i,j}\\
\ev_\hbar(T^{(r)}_{i,j})&=\hbar^{r-1}\sum_{\substack{z_1,\cdots,z_n\geq0,\\1\leq x_1,\cdots,x_n\leq p}}\limits f^m_n(z_1,\cdots, z_n)E_{i,x_1}t^{-z_1-1}\\
&\qquad\qquad\qquad\qquad\qquad E_{x_1,x_2}t^{z_1-z_2}\cdots E_{x_{n-1},x_n}t^{z_{n-1}-z_n}E_{x_n,j}t^{z_n+1}.
\end{align*}
\end{Theorem}
\begin{proof}
It is enough to show the compatibility with \eqref{ga1}-\eqref{ga3}. The compatibility with \eqref{ga1} follows from the definition of $\ev_\hbar$ and we will give the compatibility with \eqref{ga3} in the appendix. In this proof, we will give the proof of the compatibility with \eqref{ga2}. The case $i=j$ is trivial. We assume that $i\neq j$. By the definition of $\ev_\hbar$, we have
\begin{align}
&\quad[\ev_\hbar(T^{(2)}_{i,i}),\ev_\hbar(T^{(2)}_{j,j})]\nonumber\\
&=[\hbar\sum E_{i,x}t^{-s-1}E_{x,i}t^{s+1},\hbar\sum E_{j,y}t^{-v-1}E_{y,j}t^{v+1}]\nonumber\\
&=\hbar^2\sum E_{i,x}t^{-s-1}[E_{x,i}t^{s+1},E_{j,y}t^{-v-1}]E_{y,j}t^{v+1}+\hbar^2\sum E_{i,x}t^{-s-1}E_{i,y}t^{-v-1}[E_{x,i}t^{s+1},E_{y,j}t^{v+1}]\nonumber\\
&\quad+\hbar^2\sum [E_{i,x}t^{-s-1},E_{j,y}t^{-v-1}]E_{y,j}t^{v+1}E_{x,i}t^{s+1}\nonumber\\
&\quad+\hbar^2\sum E_{j,y}t^{-v-1}[E_{i,x}t^{-s-1},E_{y,j}t^{v+1}]E_{x,i}t^{s+1}.\label{al1}
\end{align}

By a direct computation, we have
\begin{align}
\eqref{al1}_1&=-\hbar^2\sum E_{i,x}t^{-s-1}(E_{j,i}t^{s-v})E_{x,j}t^{v+1}+\hbar^2\sum (s+1)E_{i,i}t^{-s-1}E_{j,j}t^{s+1},\label{al2}\\
\eqref{al1}_2&=\hbar^2\sum E_{i,x}t^{-s-1}E_{j,i}t^{-v-1}(E_{x,j}t^{s+v+2})-\hbar^2\sum E_{i,j}t^{-s-1}E_{j,y}t^{-v-1}E_{y,i}t^{s+v+2},\label{al3}\\
\eqref{al1}_3&=\hbar^2\sum (E_{i,y}t^{-s-v-2})E_{y,j}t^{v+1}E_{j,i}t^{s+1}-\hbar^2\sum E_{j,x}t^{-s-v-2}E_{i,j}t^{v+1}E_{x,i}t^{s+1},\label{al4}\\
\eqref{al1}_4&=\hbar^2\sum E_{j,x}t^{-v-1}(E_{i,j}t^{v-s})E_{x,i}t^{s+1}-\hbar^2\sum (s+1)E_{j,j}t^{-s-1}E_{i,i}t^{s+1}.\label{al5}
\end{align}
By a direct computation, we obtain
\begin{align*}
&\quad\eqref{al2}_1+\eqref{al3}_1+\eqref{al4}_1\\
&=-\hbar^2\sum E_{i,x}t^{-s-1}E_{j,i}t^{s-v}E_{x,j}t^{v+1}+\hbar^2\sum E_{i,x}t^{-s-1}E_{j,i}t^{-v-1}E_{x,j}t^{s+v+2}\\
&\quad+\hbar^2\sum E_{i,y}t^{-s-v-2}E_{y,j}t^{v+1}E_{j,i}t^{s+1}\\
&=-\hbar^2\sum E_{i,x}t^{-s-v-1}E_{j,i}t^{s}E_{x,j}t^{v+1}+\hbar^2\sum E_{i,y}t^{-s-v-2}E_{y,j}t^{v+1}E_{j,i}t^{s+1},\\
&=-\hbar^2\sum E_{i,x}t^{-v-1}E_{x,j}t^{v+1}E_{j,i}-\hbar^2\sum E_{i,x}t^{-s-v-1}[E_{j,i}t^{s},E_{x,j}t^{v+1}]\\
&=-\hbar^2\sum E_{i,x}t^{-v-1}E_{x,j}t^{v+1}E_{j,i}-\hbar^2\sum (s+1)E_{i,i}t^{-s-1}E_{j,j}t^{s+1}\\
&\quad+\hbar^2\sum (s+1)E_{i,x}t^{-s-1}E_{x,i}t^{s+1},\\
&\quad\eqref{al3}_2+\eqref{al4}_2+\eqref{al5}_1\\
&=-\hbar^2\sum E_{i,j}t^{-s-1}E_{j,y}t^{-v-1}E_{y,i}t^{s+v+2}-\hbar^2\sum E_{j,x}t^{-s-v-2}E_{i,j}t^{v+1}E_{x,i}t^{s+1}\\
&\quad+\sum E_{j,x}t^{-v-1}E_{i,j}t^{v-s}E_{x,i}t^{s+1}\\
&=-\hbar^2\sum E_{i,j}t^{-s-1}E_{j,y}t^{-v-1}E_{y,i}t^{s+v+2}+\hbar^2\sum E_{j,x}t^{-v-1}E_{i,j}t^{-s}E_{x,i}t^{s+v+1}\\
&=\hbar^2E_{i,j}\sum E_{j,y}t^{-v-1}E_{y,i}t^{v+1}+\hbar^2\sum [E_{j,x}t^{-v-1},E_{i,j}t^{-s}]E_{x,i}t^{s+v+1}\\
&=\hbar^2E_{i,j}\sum E_{j,y}t^{-v-1}E_{y,i}t^{v+1}+\hbar^2\sum (s+1)E_{j,j}t^{-s-1}E_{i,i}t^{s+1}\\
&\quad-\hbar^2(s+1)E_{i,x}t^{-s-1}E_{x,i}t^{s+1}.
\end{align*}
Thus, we obtain
\begin{align*}
&\quad[\ev_\hbar(T^{(2)}_{i,i}),\ev_\hbar(T^{(2)}_{j,j})]\nonumber\\
&=-\hbar^2\sum E_{i,x}t^{-v-1}E_{x,j}t^{v+1}E_{j,i}+\hbar^2E_{i,j}\sum E_{j,y}t^{-v-1}E_{y,i}t^{v+1}\\
&=-\hbar^2 E_{j,i}\sum E_{i,x}t^{-v-1}E_{x,j}t^{v+1}+\hbar^2\sum E_{j,y}t^{-v-1}E_{y,i}t^{v+1}E_{i,j}\\
&=-\hbar(\ev_\hbar(T^{(1)}_{j,i})\ev_\hbar(T^{(2)}_{i,j})-\ev_\hbar(T^{(2)}_{j,i})\ev_\hbar(T^{(1)}_{i,j})).
\end{align*}
\end{proof}
By Definitions~\ref{Def32} and \ref{Def33}, we find that there exists a natural homomorphism
\begin{equation*}
\kappa\colon Y_\hbar(\mathfrak{sl}(n))\to Y_{\hbar,\ve}(\widehat{\mathfrak{sl}}(n))
\end{equation*}
given by $H_{i,r}\mapsto H_{i,r},X^\pm_{i,r}\mapsto X^\pm_{i,r}$. By Lemma~\ref{Lem126} and the definition of $\iota$, $\kappa$, $\ev_{\hbar,\ve}$ and $\ev_\hbar$, we have
\begin{gather*}
\ev_{\hbar,\ve}\circ\kappa(X^\pm_{i,0})=\ev_\hbar\circ\iota(X^\pm_{i,0}),\\
\ev_{\hbar,\ve}\circ\kappa(X^\pm_{i,1})=\ev_\hbar\circ\iota(X^\pm_{i,1}).
\end{gather*}
Thus, we have
\begin{align*}
\ev_{\hbar,\ve}(X^\pm_{i,r})=\ev_\hbar\circ\iota(X^\pm_{i,r}).
\end{align*}
By using Theorem~\ref{Main}, \eqref{ee1} and \eqref{ee2}, we can write down $\ev_{\hbar,\ve}(H_{i,r})$ and $\ev_{\hbar,\ve}(X^\pm_{i,r})$ for $1\leq i\leq n-1$ and $r\geq2$.

\appendix
\section{Compatibility with \eqref{ga3}}
In the appendix, we will show the compatibility of $\ev_\hbar$ with \eqref{ga3}. We only show the case that $k=i,l=j\neq i$. The other cases can be given by the similar way by using the relation
\begin{align*}
\omega(\ev_\hbar(T^{(r)}_{i,j}))=\ev_\hbar(T^{(r)}_{i,j}),
\end{align*}
where $\omega$ is an anti-automorphism of $U(\widehat{\mathfrak{gl}}(n))$ given by $\omega(E_{i,j}t^{s})=E_{j,i}t^{-s}$ and $\omega(c)=c$.

Here after, in order to simplify the notation, we sometimes omit 
\begin{equation*}
E_{i,x_1}t^{-z_1-1}E_{x_1,x_2}t^{z_1-z_2}\cdots E_{x_{s-1},x_s}t^{z_{s-1}-z_s},\ E_{x_{s+1},x_{s+2}}t^{z_{s+1}-z_{s+2}}\cdots E_{x_{n-1},x_n}t^{z_{n-1}-z_n}E_{x_n,j}t^{z_n+1}.
\end{equation*}
We also denote $f^m_n(z_1,\cdots,z_n)$ by $f^m_n()$ and $f^m_n(z_1,\cdots, z_{i-1},a,z_{i+1},\cdots, z_n)$ by $f^m_n(z_{i+1}=a)$.

We also assume that $\hbar=1$ for the simplicity. By a direct computation, we obtain
\begin{align}
&\quad\sum f^m_n()[\sum E_{i,w}t^{-u-1}E_{w,i}t^{u+1},E_{i,x_1}t^{-z_1-1}]E_{x_1,x_2}t^{z_2-z_1}\nonumber\\
&=\sum f^m_n()E_{i,w}t^{-u-1}E_{w,x_1}t^{u-z_1}E_{x_1,x_2}t^{z_2-z_1}-\sum f^m_n()E_{i,x_1}t^{-u-1}E_{i,i}t^{u-z_1}E_{x_1,x_2}t^{z_2-z_1}\nonumber\\
&\quad+\sum f^m_n()(z_1+1)cE_{i,x_1}t^{-z_1-1}E_{x_1,x_2}t^{z_2-z_1}+\sum f^m_n()(z_1+1)E_{i,i}t^{-z_1-1}\delta_{i,x_1}E_{x_1,x_2}t^{z_2-z_1}\nonumber\\
&\quad+\sum f^m_n()E_{i,x_1}t^{-u-z_1-2}E_{i,i}t^{u+1}E_{x_1,x_2}t^{z_2-z_1}\nonumber\\
&\quad-\sum f^m_n()\delta_{i,x_1}E_{i,w}t^{-u-z_1-2}E_{w,i}t^{u+1}E_{x_1,x_2}t^{z_2-z_1},\label{A1}\\
&\quad\sum f^m_n()E_{x_{s-1},x_s}t^{z_{s-1}-z_s}[\sum E_{i,w}t^{-u-1}E_{w,i}t^{u+1},E_{x_s,x_{s+1}}t^{z_s-z_{s+1}}]E_{x_{s},x_{s+1}}t^{z_{s}-z_{s+1}}\nonumber\\
&=\sum f^m_n()\delta_{i,x_s}E_{x_{s-1},x_s}t^{z_{s-1}-z_s}E_{i,w}t^{-u-1}E_{w,x_{s+1}}t^{u+1+z_s-z_{s+1}}E_{x_{s},x_{s+1}}t^{z_{s}-z_{s+1}}\nonumber\\
&\quad-\sum f^m_n()E_{x_{s-1},x_s}t^{z_{s-1}-z_s}E_{i,x_{s+1}}t^{-u-1}E_{x_s,i}t^{u+1+z_s-z_{s+1}}E_{x_{s},x_{s+1}}t^{z_{s}-z_{s+1}}\nonumber\\
&\quad+\sum (u+1)\delta_{u+1,z_{s+1}-z_s}\delta_{i,x_s}cE_{x_{s-1},x_s}t^{z_{s-1}-z_s}E_{i,x_{s+1}}t^{-u-1}E_{x_{s},x_{s+1}}t^{z_{s}-z_{s+1}}\nonumber\\
&\quad+\sum f^m_n()(u+1)\delta_{u+1,z_{s+1}-z_s}\delta_{x_{s+1},x_s}E_{x_{s-1},x_s}t^{z_{s-1}-z_s}E_{i,i}t^{-u-1}E_{x_{s},x_{s+1}}t^{z_{s}-z_{s+1}}\nonumber\\
&\quad+\sum f^m_n()E_{x_{s-1},x_s}t^{z_{s-1}-z_s}E_{i,x_{s+1}}t^{z_s-z_{s+1}-u-1}E_{x_s,i}t^{u+1}E_{x_{s},x_{s+1}}t^{z_{s}-z_{s+1}}\nonumber\\
&\quad-\sum f^m_n()\delta_{x_{s+1},i}E_{x_{s-1},x_s}t^{z_{s-1}-z_s}E_{x_s,w}t^{z_s-z_{s+1}-u-1}E_{w,i}t^{u+1}E_{x_{s},x_{s+1}}t^{z_{s}-z_{s+1}}\nonumber\\
&\quad-\sum f^m_n()(u+1)\delta_{u+1,z_s-z_{s+1}}\delta_{i,x_{s+1}}cE_{x_s,i}t^{u+1}-\sum f^m_n()(u+1)\delta_{u+1,z_s-z_{s+1}}\delta_{x_s,x_{s+1}}E_{i,i}t^{u+1},\label{As}\\
&\quad \sum f^m_n()E_{x_{n-1},x_n}t^{z_{n-1}-z_n}[\sum E_{i,w}t^{-u-1}E_{w,i}t^{u+1},E_{x_n,j}t^{z_n+1}]\nonumber\\
&=\sum f^m_n()\delta_{i,x_n}E_{x_{n-1},x_n}t^{z_{n-1}-z_n}E_{i,w}t^{-u-1}E_{w,j}t^{u+z_n+2}\nonumber\\
&\quad-\sum f^m_n()E_{x_{n-1},x_n}t^{z_{n-1}-z_n}E_{i,j}t^{-u-1}E_{x_n,i}t^{z_n+u+2}\nonumber\\
&\quad+\sum f^m_n()E_{x_{n-1},x_n}t^{z_{n-1}-z_n}E_{i,j}t^{z_n-u}E_{x_n,i}t^{u+1}\nonumber\\
&\quad-\sum f^m_n()(z_n+1)\delta_{x_n,j}E_{x_{n-1},x_n}t^{z_{n-1}-z_n}E_{i,i}t^{u+1}.\label{An}
\end{align}
We divide \eqref{As} into two picies:
\begin{gather*}
B(s)=\eqref{As}_1+\eqref{As}_6+\eqref{As}_3+\eqref{As}_4,\\
C(s)=\eqref{As}_5+\eqref{As}_2+\eqref{As}_7+\eqref{As}_8.
\end{gather*}
At first, we compute $\eqref{A1}+B(1)$. We divide it into 
\begin{gather*}
\eqref{A1}_1+\eqref{A1}_3,\ \eqref{A1}_2+\eqref{A1}_5+B(1)_1+\eqref{A1}_4,\\
\eqref{A1}_6+B(1)_2,\ B(1)_3+B(1)_4.
\end{gather*}
As for $\eqref{A1}_1+\eqref{A1}_3$ and $B(1)_3+B(1)_4$, we obtain
\begin{align}
&\quad\eqref{A1}_1+\eqref{A1}_3\nonumber\\
&=\sum f^m_n()E_{i,w}t^{-u-1}E_{w,x_1}t^{u-z_1}E_{x_1,x_2}t^{z_1-z_2}\nonumber\\
&\quad+\sum f^m_n()(z_1+1)cE_{i,x_1}t^{-z_1-1}E_{x_1,x_2}t^{z_1-z_2},\label{e1-0}\\
&\quad B(1)_3+B(1)_4\nonumber\\
&=-\sum f^m_n(z_1=z_2+u+1)(u+1)\delta_{i,x_2}cE_{i,x_1}t^{-z_2-u-2}E_{x_1,i}t^{u+1}\nonumber\\
&\qquad+\sum f^m_n(z_2=z_1+u+1)(u+1)E_{i,x_1}t^{-z_1-1}E_{i,i}t^{-u-1}E_{x_1,x_2}t^{z_1+u+2}.\label{e1-3}
\end{align}
As for $\eqref{A1}_2+\eqref{A1}_5+B(1)_1+\eqref{A1}_4$, we obtain
\begin{align}
&\quad\eqref{A1}_2+\eqref{A1}_5\nonumber\\
&=-\sum f^m_n()E_{i,x_1}t^{-u-z_1-2}E_{i,i}t^{u+1}E_{x_1,x_2}t^{z_1-z_2}\nonumber\\
&\quad-\sum f^m_n(z_1=z_1+u)E_{i,x_1}t^{-u-1}E_{i,i}t^{-z_1}E_{x_1,x_2}t^{z_1+u-z_2}\nonumber\\
&\quad+\sum f^m_n()E_{i,x_1}t^{-u-z_1-2}E_{i,i}t^{u+1}E_{x_1,x_2}t^{z_1-z_2}\nonumber\\
&=-\sum f^m_n(z_1=z_1+u)E_{i,x_1}t^{-u-1}E_{i,i}t^{-z_1}E_{x_1,x_2}t^{z_1+u-z_2}\nonumber\\
&=-\sum f^m_n(z_1=z_1+u)E_{i,i}t^{-z_1}E_{i,x_1}t^{-u-1}E_{x_1,x_2}t^{z_1+u-z_2}\nonumber\\
&\quad-\sum f^m_n(z_1=z_1+u)(\delta_{x_1,i}E_{i,i}t^{-z_1-u-1}-\delta_{i,i}E_{i,x_1}t^{-u-z_1-1})E_{x_1,x_2}t^{z_1+u-z_2}\nonumber\\
&=-\sum f^m_n()E_{i,i}E_{i,x_1}t^{-z_1-1}E_{x_1,x_2}t^{z_1+u-z_2}\nonumber\\
&\quad-\sum f^m_n(z_1=z_1+u+1)E_{i,i}t^{-z_1-1}E_{i,x_1}t^{-u-1}E_{x_1,x_2}t^{z_1+u+1-z_2}\nonumber\\
&\quad-\sum f^m_n()(z_1+1)(\delta_{x_1,i}E_{i,i}t^{-z_1-1}-\delta_{i,i}E_{i,x_1}t^{-z_1-1})E_{x_1,x_2}t^{z_1-z_2}.
\end{align}
Thus, we have
\begin{align}
&\quad\eqref{A1}_2+\eqref{A1}_5+B(1)_1+\eqref{A1}_4\nonumber\\
&=-E_{i,i}\sum f^m_n(z_1=u)E_{i,x_1}t^{-u-1}E_{x_1,x_2}t^{u-z_2}\nonumber\\
&\qquad+\sum (f^m_n()-f^m_n(z_1=z_1+u+1))E_{i,i}t^{-z_1-1}E_{i,w}t^{-u-1}(E_{w,x_{2}}t^{u+1+z_1-z_{2}})\nonumber\\
&\quad+\sum f^m_n()(z_1+1)(E_{i,x_1}t^{-z_1-1})E_{x_1,x_2}t^{z_1-z_2}\nonumber\\
&=-E_{i,i}\sum f^m_n(z_1=u)E_{i,x_1}t^{-u-1}E_{x_1,x_2}t^{u-z_2}\nonumber\\
&\quad-\sum f^m_{n+1}(z_0=z_1+u+1)(u+1)cE_{i,i}t^{-z_1-1}E_{i,w}t^{-u-1}(E_{w,x_{2}}t^{u+1+z_1-z_{2}})\nonumber\\
&\quad+\sum f^m_n()(z_1+1)(E_{i,x_1}t^{-z_1-1})E_{x_1,x_2}t^{z_1-z_2}.\label{e1-1}
\end{align}
As for $\eqref{A1}_6+B(1)_2$, we have
\begin{align*}
&\quad\eqref{A1}_6+B(1)_2\nonumber\\
&=-\sum f^m_n() E_{i,w}t^{-u-z_1-2}E_{w,i}t^{u+1}E_{i,x_2}t^{z_1-z_2}\nonumber\\
&\quad+\sum f^m_n()E_{i,x_1}t^{-z_1-1}E_{i,x_2}t^{z_1-z_2-u-1}E_{x_1,i}t^{u+1}\nonumber\\
&=-\sum f^m_n() E_{i,w}t^{-u-z_1-2}E_{i,x_2}t^{z_1-z_2}E_{w,i}t^{u+1}\nonumber\\
&\quad-\sum f^m_n() E_{i,w}t^{-u-z_1-2}[E_{w,i}t^{u+1},E_{i,x_2}t^{z_1-z_2}]\nonumber\\
&\quad+\sum f^m_n(z_1=z_1+u+1)E_{i,x_1}t^{-z_1-u-2}E_{i,x_2}t^{z_1-z_2}E_{x_1,i}t^{u+1}\nonumber\\
&\quad+\sum f^m_n()E_{i,x_1}t^{-z_1-1}E_{i,x_2}t^{-z_2-u-1}E_{x_1,i}t^{u+z_1+1}\nonumber\\
&=\sum (f^m_n(z_1=z_1+u+1)-f^m_n())E_{i,x_1}t^{-z_1-u-2}E_{i,x_2}t^{z_1-z_2}E_{x_1,i}t^{u+1}\nonumber\\
&\quad-\sum f^m_n() E_{i,w}t^{-u-z_1-2}[E_{w,i}t^{u+1},E_{i,x_2}t^{z_1-z_2}]\nonumber\\
&\quad+\sum f^m_n()E_{i,x_1}t^{-z_1-1}E_{i,x_2}t^{-z_2-u-1}E_{x_1,i}t^{u+z_1+1}.
\end{align*}
Since we obtain
\begin{align*}
&\quad\sum (f^m_n(z_1=z_1+u+1)-f^m_n())E_{i,x_1}t^{-z_1-u-2}E_{i,x_2}t^{z_1-z_2}E_{x_1,i}t^{u+1}\nonumber\\
&=\sum (f^m_n(z_1=z_1+u+1)-f^m_n())E_{i,x_1}t^{-z_1-u-2}E_{x_1,i}t^{u+1}E_{i,x_2}t^{z_1-z_2}\nonumber\\
&\quad-\sum (f^m_n(z_1=z_1+u+1)-f^m_n()) E_{i,w}t^{-u-z_1-2}((u+1)\delta_{u+1,z_2-z_1}\delta_{w,x_2}c)\nonumber\\
&\quad+\sum (f^m_n(z_1=z_1+u+1)-f^m_n()) E_{i,w}t^{-u-z_1-2})(\delta_{w,x_2}E_{i,i}t^{z_1-z_2+u+1})\nonumber\\
&\quad-\sum (f^m_n(z_1=z_1+u+1)-f^m_n()) E_{i,w}t^{-u-z_1-2})E_{w,x_2}t^{z_1-z_2+u+1})\nonumber\\
&\quad-\sum (f^m_n(z_1=z_1+u+1)-f^m_n()) E_{i,w}t^{-u-z_1-2}((u+1)\delta_{u+1,z_2-z_1}\delta_{w,i}\delta_{x_2,i})
\end{align*}
and
\begin{align*}
&\quad-\sum f^m_n() E_{i,w}t^{-u-z_1-2}[E_{w,i}t^{u+1},E_{i,x_2}t^{z_1-z_2}]\nonumber\\
&=-\sum f^m_n() E_{i,w}t^{-u-z_1-2}(E_{w,x_2}t^{z_1-z_2+u+1})+\sum f^m_n() E_{i,w}t^{-u-z_1-2}(\delta_{w,x_2}E_{i,i}t^{z_1-z_2+u+1})\nonumber\\
&\quad-\sum f^m_n() E_{i,w}t^{-u-z_1-2}((u+1)\delta_{u+1,z_2-z_1}\delta_{w,x_2}c)\nonumber\\
&\quad-\sum f^m_n() E_{i,w}t^{-u-z_1-2}((u+1)\delta_{u+1,z_2-z_1}\delta_{w,i}\delta_{x_2,i})
\end{align*}
by a direct computation, we obtain
\begin{align}
&\quad\eqref{A1}_6+B(1)_2\nonumber\\
&=\sum (f^m_n(z_1=z_1+u+1)-f^m_n()) E_{i,w}t^{-u-z_1-2}E_{w,i}t^{u+1}E_{i,x_2}t^{z_1-z_2}\nonumber\\
&\quad+\sum (f^m_n(z_1=z_1+u+1)-f^m_n()) E_{i,w}t^{-u-z_1-2}(\delta_{w,x_2}E_{i,i}t^{z_1-z_2+u+1})\nonumber\\
&\quad-\sum (f^m_n(z_1=z_1+u+1)-f^m_n()) E_{i,w}t^{-u-z_1-2}(E_{w,x_2}t^{z_1-z_2+u+1})\nonumber\\
&\quad-\sum (f^m_n(z_1=z_1+u+1)-f^m_n()) E_{i,w}t^{-u-z_1-2}((u+1)\delta_{u+1,z_2-z_1}\delta_{w,x_2}c)\nonumber\\
&\quad-\sum (f^m_n(z_1=z_1+u+1)-f^m_n()) E_{i,w}t^{-u-z_1-2}((u+1)\delta_{u+1,z_2-z_1}\delta_{w,i}\delta_{x_2,i})\nonumber\\
&\quad-\sum f^m_n() E_{i,w}t^{-u-z_1-2}E_{w,x_2}t^{z_1-z_2+u+1}+\sum f^m_n() E_{i,w}t^{-u-z_1-2}(\delta_{w,x_2}E_{i,i}t^{z_1-z_2+u+1})\nonumber\\
&\quad-\sum f^m_n()E_{i,w}t^{-u-z_1-2}((u+1)\delta_{u+1,z_2-z_1}\delta_{w,x_2}c)\nonumber\\
&\quad-\sum f^m_n() E_{i,w}t^{-u-z_1-2}((u+1)\delta_{u+1,z_2-z_1}\delta_{w,i}\delta_{x_2,i})\nonumber\\
&\quad+\sum f^m_n()E_{i,x_1}t^{-z_1-1}E_{i,x_2}t^{-z_2-u-1}E_{x_1,i}t^{u+z_1+1}\nonumber\\
&=\sum f^m_{n+1}(z_0=z_1+u+1) (u+1)cE_{i,w}t^{-u-z_1-2}E_{w,i}t^{u+1}E_{i,x_2}t^{z_1-z_2}\nonumber\\
&\quad+\sum f^m_n(z_1=z_1+u+1) E_{i,x_2}t^{-u-z_1-2}(E_{i,i}t^{z_1-z_2+u+1})\nonumber\\
&\quad-\sum f^m_n(z_1=z_1+u+1) E_{i,w}t^{-u-z_1-2}(E_{w,x_2}t^{z_1-z_2+u+1})\nonumber\\
&\quad-\sum f^m_n(z_1=z_1+u+1) E_{i,x_2}t^{-u-z_1-2}((u+1)\delta_{u+1,z_2-z_1}c)\nonumber\\
&\quad-\sum f^m_n(z_1=z_1+u+1) E_{i,i}t^{-u-z_1-2}((u+1)\delta_{u+1,z_2-z_1}\delta_{x_2,i})\nonumber\\
&\quad+\sum f^m_n()E_{i,x_1}t^{-z_1-1}E_{i,x_2}t^{-z_2-u-1}E_{x_1,i}t^{u+z_1+1}.\label{e1-2}
\end{align}
Next, we compute $B(s+1)+C(s)$. We divide it into 
\begin{align*}
B(s+1)_1+C(s)_2,\ B(s+1)_4+C(s)_1, B(s+1)_2+C(s)_3,\ B(s+1)_3+C(s)_4.
\end{align*}
By a direct computation, we obtain
\begin{align}
&\quad B(s+1)_2+C(s)_3\nonumber\\
&=\sum f^m_n(z_{s+1}=z_s+u+1)E_{i,x_{s+1}}t^{-u-1}((u+1)\delta_{i,x_s}c)E_{x_{s+1},x_{s+2}}t^{z_s+u+1-z_{s+2}}\nonumber\\
&=-\sum f^m_n()E_{x_{s},x_{s+1}}t^{z_{s}-z_{s+1}}((u+1)\delta_{u+1,z_{s+1}-z_{s+2}}\delta_{i,x_{s+2}}c)E_{x_{s+1},i}t^{u+1},\label{e2-3}\\
&\quad B(s+1)_3+C(s)_4\nonumber\\
&=\sum f^m_n(z_{s+1}=z_s+u+1)E_{i,i}t^{-u-1}((u+1)\delta_{u+1,z_{s+1}-z_s})E_{x_{s},x_{s+2}}t^{z_{s+1}-z_{s+2}}\nonumber\\
&\quad-\sum f^m_n()E_{x_{s},x_{s+2}}t^{z_{s}-z_{s+1}}((u+1)\delta_{u+1,z_{s+1}-z_{s+2}})E_{i,i}t^{u+1}.\label{e2-4}.
\end{align}
Since we obtain
\begin{align*}
&\quad-\sum f^m_n()E_{i,x_{s+1}}t^{-u-1}(E_{x_s,i}t^{u+1+z_s-z_{s+1}})E_{x_{s+1},x_{s+2}}t^{z_{s+1}-z_{s+2}}\nonumber\\
&=-\sum f^m_n(z_{s+1}=z_{s+1}+u+1)E_{i,x_{s+1}}t^{-u-1}(E_{x_s,i}t^{z_s-z_{s+1}})E_{x_{s+1},x_{s+2}}t^{z_{s+1}+u+1-z_{s+2}}\nonumber\\
&\quad-\sum f^m_n()E_{i,x_{s+1}}t^{-u-z_{s+1}-1}(E_{x_s,i}t^{u+1+z_s})E_{x_{s+1},x_{s+2}}t^{z_{s+1}-z_{s+2}}\nonumber\\
\end{align*}
and
\begin{align*}
&\quad\sum f^m_n()E_{x_{s},i}t^{z_{s}-z_{s+1}}E_{i,w}t^{-u-1}(E_{w,x_{s+2}}t^{u+1+z_{s+1}-z_{s+2}})\nonumber\\
&=\sum f^m_n()E_{i,w}t^{-u-1}E_{x_{s},i}t^{z_{s}-z_{s+1}}(E_{w,x_{s+2}}t^{u+1+z_{s+1}-z_{s+2}})\nonumber\\
&\quad+\sum f^m_n()(\delta_{i,i}E_{x_s,w}t^{z_s-z_{s+1}-u-1})(E_{w,x_{s+2}}t^{u+1+z_{s+1}-z_{s+2}})\nonumber\\
&\quad-\sum f^m_n()(E_{i,i}t^{z_s-z_{s+1}-u-1})(E_{x_s,x_{s+2}}t^{u+1+z_{s+1}-z_{s+2}})\nonumber\\
&\quad+\sum f^m_n()((u+1)\delta_{u+1,z_s-z_{s+1}}c)(E_{x_s,x_{s+2}}t^{u+1+z_{s+1}-z_{s+2}})\nonumber\\
&\quad+\sum f^m_n()((u+1)\delta_{u+1,z_s-z_{s+1}}\delta_{x_s,i})(E_{i,x_{s+2}}t^{u+1+z_{s+1}-z_{s+2}}),
\end{align*}
we have
\begin{align*}
&\quad B(s+1)_1+C(s)_2\\
&=-\sum f^m_n()E_{i,x_{s+1}}t^{-u-z_{s+1}-1}(E_{x_s,i}t^{u+1+z_s})E_{x_{s+1},x_{s+2}}t^{z_{s+1}-z_{s+2}}\nonumber\\
&\quad+\sum (f^m_n()-f^m_n(z_{s+1}=z_{s+1}+u+1))E_{i,w}t^{-u-1}E_{x_{s},i}t^{z_{s}-z_{s+1}}(E_{w,x_{s+2}}t^{u+1+z_{s+1}-z_{s+2}})\nonumber\\
&\quad+\sum f^m_n()(E_{x_s,w}t^{z_s-z_{s+1}-u-1})(E_{w,x_{s+2}}t^{u+1+z_{s+1}-z_{s+2}})\nonumber\\
&\quad-\sum f^m_n()(E_{i,i}t^{z_s-z_{s+1}-u-1})(E_{x_s,x_{s+2}}t^{u+1+z_{s+1}-z_{s+2}})\nonumber\\
&\quad+\sum f^m_n()((u+1)\delta_{u+1,z_s-z_{s+1}}c)(E_{x_s,x_{s+2}}t^{u+1+z_{s+1}-z_{s+2}})\nonumber\\
&\quad+\sum f^m_n()((u+1)\delta_{u+1,z_s-z_{s+1}}\delta_{x_s,i})(E_{i,x_{s+2}}t^{u+1+z_{s+1}-z_{s+2}}),
\end{align*}
Since we obtain
\begin{align*}
&\quad\sum (f^m_n()-f^m_n(z_{s+1}=z_{s+1}+u+1))E_{i,w}t^{-u-1}E_{x_{s},i}t^{z_{s}-z_{s+1}}(E_{w,x_{s+2}}t^{u+1+z_{s+1}-z_{s+2}})\nonumber\\
&=\sum (f^m_n()-f^m_n(z_{s+1}=z_{s+1}+u+1))E_{x_{s},i}t^{z_{s}-z_{s+1}}E_{i,w}t^{-u-1}(E_{w,x_{s+2}}t^{u+1+z_{s+1}-z_{s+2}})\nonumber\\
&\quad-\sum (f^m_n()-f^m_n(z_{s+1}=z_{s+1}+u+1))(E_{x_s,w}t^{z_s-z_{s+1}-u-1})(E_{w,x_{s+2}}t^{u+1+z_{s+1}-z_{s+2}})\nonumber\\
&\quad+\sum (f^m_n()-f^m_n(z_{s+1}=z_{s+1}+u+1))(E_{i,i}t^{z_s-z_{s+1}-u-1})(E_{x_s,x_{s+2}}t^{u+1+z_{s+1}-z_{s+2}})\nonumber\\
&\quad-\sum (f^m_n()-f^m_n(z_{s+1}=z_{s+1}+u+1))((u+1)\delta_{u+1,z_s-z_{s+1}}c)(E_{x_s,x_{s+2}}t^{u+1+z_{s+1}-z_{s+2}})\nonumber\\
&\quad-\sum (f^m_n()-f^m_n(z_{s+1}=z_{s+1}+u+1))((u+1)\delta_{u+1,z_s-z_{s+1}}\delta_{x_s,i})(E_{i,x_{s+2}}t^{u+1+z_{s+1}-z_{s+2}}),
\end{align*}
we have
\begin{align}
&\quad B(s+1)_1+C(s)_2\nonumber\\
&=-\sum f^m_n()E_{i,x_{s+1}}t^{-u-z_{s+1}-1}(E_{x_s,i}t^{u+1+z_s})E_{x_{s+1},x_{s+2}}t^{z_{s+1}-z_{s+2}}\nonumber\\
&\quad+\sum (f^m_n()-f^m_n(z_{s+1}=z_{s+1}+u+1))E_{x_{s},i}t^{z_{s}-z_{s+1}}E_{i,w}t^{-u-1}(E_{w,x_{s+2}}t^{u+1+z_{s+1}-z_{s+2}})\nonumber\\
&\quad-\sum (f^m_n()-f^m_n(z_{s+1}=z_{s+1}+u+1))(E_{x_s,w}t^{z_s-z_{s+1}-u-1})(E_{w,x_{s+2}}t^{u+1+z_{s+1}-z_{s+2}})\nonumber\\
&\quad+\sum (f^m_n()-f^m_n(z_{s+1}=z_{s+1}+u+1))(E_{i,i}t^{z_s-z_{s+1}-u-1})(E_{x_s,x_{s+2}}t^{u+1+z_{s+1}-z_{s+2}})\nonumber\\
&\quad-\sum (f^m_n()-f^m_n(z_{s+1}=z_{s+1}+u+1))((u+1)\delta_{u+1,z_s-z_{s+1}}c)(E_{x_s,x_{s+2}}t^{u+1+z_{s+1}-z_{s+2}})\nonumber\\
&\quad-\sum (f^m_n()-f^m_n(z_{s+1}=z_{s+1}+u+1))((u+1)\delta_{u+1,z_s-z_{s+1}}\delta_{x_s,i})(E_{i,x_{s+2}}t^{u+1+z_{s+1}-z_{s+2}})\nonumber\\
&\quad+\sum f^m_n()(\delta_{i,i}E_{x_s,w}t^{z_s-z_{s+1}-u-1})(E_{w,x_{s+2}}t^{u+1+z_{s+1}-z_{s+2}})\nonumber\\
&\qquad-\sum f^m_n()(E_{i,i}t^{z_s-z_{s+1}-u-1})(E_{x_s,x_{s+2}}t^{u+1+z_{s+1}-z_{s+2}})\nonumber\\
&\qquad+\sum f^m_n()((u+1)\delta_{u+1,z_s-z_{s+1}}\delta_{i,i}c)(E_{x_s,x_{s+2}}t^{u+1+z_{s+1}-z_{s+2}})\nonumber\\
&\qquad+\sum f^m_n()((u+1)\delta_{u+1,z_s-z_{s+1}}\delta_{x_s,i})(E_{i,x_{s+2}}t^{u+1+z_{s+1}-z_{s+2}})\nonumber\\
&=-\sum f^m_n()E_{i,x_{s+1}}t^{-u-z_{s+1}-1}(E_{x_s,i}t^{u+1+z_s})E_{x_{s+1},x_{s+2}}t^{z_{s+1}-z_{s+2}}\nonumber\\
&\quad+\sum f^m_{n+1}(z_{0}=z_{s+1}+u+1)(u+1)cE_{x_{s},i}t^{z_{s}-z_{s+1}}E_{i,w}t^{-u-1}(E_{w,x_{s+2}}t^{u+1+z_{s+1}-z_{s+2}})\nonumber\\
&\quad+\sum f^m_n(z_{s+1}=z_{s+1}+u+1)(E_{x_s,w}t^{z_s-z_{s+1}-u-1})(E_{w,x_{s+2}}t^{u+1+z_{s+1}-z_{s+2}})\nonumber\\
&\quad-\sum f^m_n(z_{s+1}=z_{s+1}+u+1)(E_{i,i}t^{z_s-z_{s+1}-u-1})(E_{x_s,x_{s+2}}t^{u+1+z_{s+1}-z_{s+2}})\nonumber\\
&\quad+\sum f^m_n(z_{s+1}=z_{s+1}+u+1)((u+1)\delta_{u+1,z_s-z_{s+1}}c)(E_{x_s,x_{s+2}}t^{u+1+z_{s+1}-z_{s+2}})\nonumber\\
&\quad+\sum f^m_n(z_{s+1}=z_{s+1}+u+1)((u+1)\delta_{u+1,z_s-z_{s+1}}\delta_{x_s,i})(E_{i,x_{s+2}}t^{u+1+z_{s+1}-z_{s+2}})\nonumber\\
&=-\sum f^m_n()E_{i,x_{s+1}}t^{-u-z_{s+1}-1}(E_{x_s,i}t^{u+1+z_s})E_{x_{s+1},x_{s+2}}t^{z_{s+1}-z_{s+2}}\nonumber\\
&\quad+\sum f^m_{n+1}(z_{0}=z_{s+1}+u+1)(u+1)cE_{x_{s},i}t^{z_{s}-z_{s+1}}E_{i,w}t^{-u-1}(E_{w,x_{s+2}}t^{u+1+z_{s+1}-z_{s+2}})\nonumber\\
&\quad+\sum f^m_n()z_{s+1}(E_{x_s,w}t^{z_s-z_{s+1}})(E_{w,x_{s+2}}t^{z_{s+1}-z_{s+2}})\nonumber\\
&\quad-\sum f^m_n()z_{s+1}(E_{i,i}t^{z_s-z_{s+1}})(E_{x_s,x_{s+2}}t^{z_{s+1}-z_{s+2}})\nonumber\\
&\quad+\sum f^m_n(z_{s+1}=z_{s+1}+u+1)((u+1)\delta_{u+1,z_s-z_{s+1}}c)(E_{x_s,x_{s+2}}t^{u+1+z_{s+1}-z_{s+2}})\nonumber\\
&\quad+\sum f^m_n(z_{s+1}=z_{s+1}+u+1)((u+1)\delta_{u+1,z_s-z_{s+1}}\delta_{x_s,i})(E_{i,x_{s+2}}t^{u+1+z_{s+1}-z_{s+2}}).\label{e2-1}
\end{align}
Similarly, we obtain
\begin{align}
&\quad B(s+1)_4+C(s)_1\nonumber\\
&=\sum (f^m_{n+1}(z_{0}=z_{s+1}+u+1))(u+1)cE_{x_{s},x_{s+1}}t^{z_{s}-z_{s+1}-u-1}E_{x_{s+1},i}t^{u+1}(E_{i,x_{s+2}}t^{z_{s+1}-z_{s+2}})\nonumber\\
&\quad+\sum f^m_n()z_{s+1}E_{x_{s},x_{s+2}}t^{z_{s}-z_{s+1}}(E_{i,i}t^{z_{s+1}-z_{s+2}})\nonumber\\
&\quad-\sum f^m_n()z_{s+1}E_{x_{s},x_{s+1}}t^{z_{s}-z_{s+1}}(E_{x_{s+1},x_{s+2}}t^{z_{s+1}-z_{s+2}})\nonumber\\
&\quad-\sum f^m_n(z_{s+1}=z_{s+1}+u+1)E_{x_{s},x_{s+2}}t^{z_{s}-z_{s+1}-u-1}((u+1)\delta_{u+1,z_{s+2}-z_{s+1}}c)\nonumber\\
&\quad-\sum f^m_n(z_{s+1}=z_{s+1}+u+1)E_{x_{s},i}t^{z_{s}-z_{s+1}-u-1}((u+1)\delta_{u+1,z_{s+2}-z_{s+1}}\delta_{x_{s+2},i})\nonumber\\
&\quad+\sum f^m_n()E_{x_{s},x_{s+1}}t^{z_{s}-z_{s+1}}(E_{i,x_{s+2}}t^{-z_{s+2}-u-1})E_{x_{s+1},i}t^{z_{s+1}+u+1}.\label{e2-2}
\end{align}
Finally, we compute $\eqref{An}+C(n-1)$. We divide it into 
\begin{align*}
\eqref{An}_1+C(n-1)_1,\ \eqref{An}_2+\eqref{An}_3+C(n-1)_4,\ \eqref{An}_4+C(n-1)_2+C(n-1)_3
\end{align*}
Similarly, we obtain
\begin{align}
&\quad\eqref{An}_1+C(n-1)_1\nonumber\\
&=-\sum f^m_{n+1}(z_0=z_n+u+1)(u+1)cE_{x_{n-1},i}t^{z_{n-1}-z_n}E_{i,w}t^{-u-1}(E_{w,j}t^{u+z_n+2})\nonumber\\
&\quad-\sum f^m_n()z_n(E_{i,i}t^{z_{n-1}-z_n})(E_{x_{n-1},j}t^{z_n+1})\nonumber\\
&\quad+\sum f^m_n()z_n(E_{x_{n-1},w}t^{z_{n-1}-z_n})(E_{w,j}t^{z_n+1})\nonumber\\
&\quad+\sum f^m_n(z_n=z_n+u+1)((u+1)\delta_{u+1,z_{n-1}-z_n}c)(E_{x_{n-1},j}t^{u+z_n+2})\nonumber\\
&\quad+\sum f^m_n(z_n=z_n+u+1)((u+1)\delta_{u+1,z_{n-1}-z_n}\delta_{x_{n-1},i})(E_{i,j}t^{u+z_n+2})\nonumber\\
&\quad-\sum f^m_n()E_{i,x_n}t^{-u-z_n-1}(E_{x_{n-1},i}t^{u+1+z_{n-1}})E_{x_n,j}t^{z_n+1},\label{e3-1}\\
&\quad\eqref{An}_2+\eqref{An}_3+\eqref{An}_4+C(n-1)_4\nonumber\\
&=\sum f^m_{n+1}(z_0=z_n+u+1)(u+1)c(E_{x_{n-1},w}t^{z_{n-1}-z_n-u-1})E_{w,i}t^{u+1}E_{i,j}t^{z_n+1}\nonumber\\
&\qquad+\sum f^m_n()E_{x_{n-1},x_n}t^{z_{n-1}-z_n}E_{x_n,i}t^{z_n+1}(E_{i,j})\nonumber\\
&\quad-\sum f^m_n()(z_n+1)E_{x_{n-1},x_n}t^{z_{n-1}-z_n}(E_{x_n,j}t^{z_n+1}),\label{e3-2}\\
&\quad C(n-1)_2+C(n-1)_3\nonumber\\
&=\sum f^m_n()E_{i,x_n}t^{-u-1}((u+1)\delta_{u+1,z_{n-1}-z_n}\delta_{i,x_{n-1}}c)E_{x_n,j}^{z_{n-1}+u+2}\nonumber\\
&\quad+\sum f^m_n()((u+1)\delta_{u+1,z_{n-1}-z_n})E_{i,i}t^{u+1}E_{x_{n-1},j}t^{z_n+1}.\label{e3-4}
\end{align}

Similarly, by a direct computation, we obtain
\begin{align}
&\quad[\sum E_{i,w}t^{-u-1}E_{w,i}t^{u+1},[\sum f^m_1()E_{i,w}t^{-z_1-1}E_{w,j}t^{z_1+1}]\nonumber\\
&=\sum f^m_1()E_{i,w}t^{-u-1}(E_{w,x_1}t^{u-z_1})E_{x_1,j}t^{z_1+1}\nonumber\\
&\quad-\sum f^m_1()E_{i,i}E_{i,x_1}t^{-z_1-1}E_{x_1,j}t^{z_1+1}\nonumber\\
&\quad-\sum f^m_2(z_1=z_2+u+1)(u+1)cE_{i,i}t^{-z_1-1}E_{i,x_1}t^{-u-1}E_{x_1,j}t^{z_1+u+1}\nonumber\\
&\quad+\sum f^m_1()(z_1+1)cE_{i,x_1}t^{-z_1-1}E_{x_1,j}t^{z_1+1}\nonumber\\
&\quad+\sum f^m_1()E_{i,x_1}t^{-1-z_1}E_{x_1,i}t^{z_1+1}E_{i,j}\nonumber\\
&\quad+\sum f^m_1(z_2=z_1+u+1)(u+1)cE_{i,x_1}t^{-1-z_1-u}E_{x_1,i}t^{u+1}E_{i,j}t^{z_1+1}.\label{e4-1}
\end{align}
Then, we obtain
\begin{align}
&\quad\sum_{n=2}^m\eqref{e1-0}+\eqref{e4-1}_1+\eqref{e4-1}_4\nonumber\\
&=\sum_{n=1}^{m+1}f^{m+1}_n()E_{i,x_1}t^{z_1-z_2}E_{x_1,x_2}t^{z_2-z_3}\cdots E_{x_{n-1},x_n}t^{z_{n-1}-z_n}E_{x_n,j}t^{z_n+1},\\
&\quad\sum_{n=2}^m\eqref{e1-1}_1+\eqref{e4-1}_2\nonumber\\
&=-E_{i,i}\sum_{n=1}^{m}f^{m}_n()E_{i,x_1}t^{z_1-z_2}E_{x_1,x_2}t^{z_2-z_3}\cdots E_{x_{n-1},x_n}t^{z_{n-1}-z_n}E_{x_n,j}t^{z_n+1},\\
&\quad\sum_{n=2}^m\eqref{e3-2}_2+\eqref{e4-1}_5\nonumber\\
&=\sum_{n=1}^{m}f^{m}_n()E_{i,x_1}t^{z_1-z_2}E_{x_1,x_2}t^{z_2-z_3}\cdots E_{x_{n-1},x_n}t^{z_{n-1}-z_n}E_{x_n,i}t^{z_n+1}E_{i,j},\\
&\quad\eqref{e1-1}_2+\sum_{n=2}^m(\eqref{e1-3}_1+\eqref{e1-2}_1+\sum_{s=1}^{n-1}(\eqref{e2-3}_1+\eqref{e2-1}_2)+\eqref{e3-1}_1+\eqref{e3-4}_1)+\eqref{e4-1}_3=0,\\
&\quad\sum_{n=2}^m(\eqref{e1-1}_3+\eqref{e1-2}_3+\sum_{s=1}^{n-1}(\eqref{e2-1}_3+\eqref{e2-2}_3)+\eqref{e3-1}_3+\eqref{e3-2}_3)=0,\\
&\quad\sum_{n=2}^m(\eqref{e1-3}_2+\sum_{s=1}^{n-1}(\eqref{e2-4})+\eqref{e3-4}_2
=0,\\
&\quad\sum_{n=2}^m(\eqref{e1-2}_2+\sum_{s=1}^{n-1}(\eqref{e2-1}_4+\eqref{e2-2}_2)+\eqref{e3-1}_2=0,\\
&\quad\eqref{e1-2}_6+\sum_{s=1}^{n-1}(\eqref{e2-1}_1+\eqref{e2-2}_6)+\eqref{e3-1}_6=0,\\
&\quad\sum_{n=2}^m(\eqref{e1-2}_4+\sum_{s=1}^{n-1}(\eqref{e2-1}_5+\eqref{e2-2}_4)+\eqref{e3-1}_4)=0,\\
&\quad\sum_{n=2}^m(\eqref{e1-2}_5+\sum_{s=1}^{n-1}(\eqref{e2-1}_6+\eqref{e2-2}_7)+\eqref{e3-1}_5)=0,\\
&\quad\sum_{n=2}^m(\sum_{s=1}^{n-1}(\eqref{e2-3}_2+\eqref{e2-2}_1)+\eqref{e3-2}_1+\eqref{e4-1}_6=0.
\end{align}
\section*{Statements and Declarations}
\subsection*{Data availability}
The authors confirm that the data supporting the findings of this study are available within the article and its supplementary materials.
\subsection*{Conflicts of interest}
The authors declare no conflicts of interest associated with this manuscript.
\subsection*{Ethical Approval declaration}
Ethical approval was not required for this study as it did not involve human or animal subjects.


\begin{thebibliography}{10}
\bibitem{D1}
V.~G. Drinfeld.
\newblock Hopf algebras and the quantum {Y}ang-{B}axter equation.
\newblock {\em Dokl. Akad. Nauk SSSR}, 283(5):1060--1064, 1985.

\bibitem{D2}
V.~G. Drinfeld.
\newblock A new realization of {Y}angians and of quantum affine algebras.
\newblock {\em Dokl. Akad. Nauk SSSR}, 296(1):13--17, 1987.

\bibitem{Gu2}
N.~Guay.
\newblock Cherednik algebras and {Y}angians.
\newblock {\em Int. Math. Res. Not.}, (57):3551--3593, 2005.

\bibitem{Gu1}
N.~Guay.
\newblock Affine {Y}angians and deformed double current algebras in type {A}.
\newblock {\em Adv. Math.}, 211(2):436--484, 2007.

\bibitem{GNW}
N.~Guay, H.~Nakajima, and C.~Wendlandt.
\newblock Coproduct for {Y}angians of affine {K}ac-{M}oody algebras.
\newblock {\em Adv. Math.}, 338:865--911, 2018.

\bibitem{K1}
R.~Kodera.
\newblock On {G}uay's evaluation map for affine {Y}angians.
\newblock {\em Algebr. Represent. Theory}, 24(1):253--267, 2021.


\bibitem{KU}
R.~Kodera and M.~Ueda.
\newblock Coproduct for affine {Y}angians and parabolic induction for
  rectangular {$W$}-algebras.
\newblock {\em Lett. Math. Phys.}, 112(1):Paper No. 3, 37, 2022.

\bibitem{MNT}
A.~Matsuo, K.~Nagatomo, and A.~Tsuchiya.
\newblock Quasi-finite algebras graded by {H}amiltonian and vertex operator
  algebras.
\newblock {\em London Math. Soc. Lecture Note Ser.}, 372:282--329, 2010.

\bibitem{U2}
M.~Ueda.
\newblock Affine super {Y}angians and rectangular {$W$}-superalgebras.
\newblock {\em J. Math. Phys.}, 63(5):Paper No. 051701, 34, 2022.

\end{thebibliography}
\end{document}